%% file: AH_long.tex
\documentclass[11pt,a4paper]{article}

\input{head.tex}

\begin{document}

\input{title_abstract.tex}

\input{Sec191122.tex}

\input{Apdx190814.tex}

\input{bibliography.tex}
\end{document}

%% file: head.tex
\usepackage[noblocks]{authblk}
\usepackage{etoolbox}
\usepackage{comment,amsmath,amssymb,amsthm,bbm,mathtools,mathrsfs}
\usepackage{tikz,pgfplots,color,setspace,lmodern}
\usetikzlibrary{shapes, backgrounds,patterns,decorations.pathreplacing}
\usepgfplotslibrary{fillbetween}
\usepackage{multicol,float}
\usepackage{makeidx}
\usepackage{hyperref, enumerate}
\pgfplotsset{compat=newest}
\makeindex

\def\vol{{\rm Vol}}

\newcommand{\p} {\textnormal{\textsf{P}}}

\newcommand{\dd}{\mathrm{d}}

\newcommand{\R} {\mathbb{R}}
\newcommand{\Q} {\mathbb{Q}}
\newcommand{\Z} {\mathbb{Z}}

\newcommand{\nn}{\nonumber}

\newcommand{\dist} {\textnormal{dist}}

\newcommand*{\rom}[1]{\text{\expandafter\@slowromancap\romannumeral #1@}}
\newcommand{\dom}{\textnormal{dom}}

\providecommand{\abs}[1]{\lvert#1\rvert}
\providecommand{\Abs}[1]{\Bigr\lvert#1\Bigl\rvert}
\providecommand{\norm}[1]{\lVert#1\rVert}
\providecommand{\Norm}[1]{\Bigr\lVert#1\Bigl\rVert}
\providecommand{\floor}[1]{\lfloor#1\rfloor}

\providecommand{\mb}[1]{\mathbb{#1}}
\providecommand{\mc}[1]{\mathcal{#1}}
\providecommand{\ms}[1]{\mathscr{#1}}
\providecommand{\lbar}[1]{\underline{#1}}

\DeclareMathOperator*{\osc}{osc}

\newtheorem{theorem}{Theorem}
\newtheorem{lemma}[theorem]{Lemma}
\newtheorem{corollary}[theorem]{Corollary}

\newtheorem{remark}[theorem]{Remark}

%% file: title_abstract.tex
\title{Quenched local central limit theorem for random walks in a time-dependent balanced random environment}

\author{Jean-Dominique Deuschel%
  \thanks{Electronic address: \texttt{deuschel@math.tu-berlin.de}}}
\affil{Instit\"ut f\"ur Mathematik\\ Technische Universit\"at Berlin}

\author{Xiaoqin Guo%
  \thanks{Electronic address: \texttt{xguo@math.wisc.edu}}}
\affil{Department of Mathematics\\University of Wisconsin-Madison}
\maketitle

\begin{abstract}
We prove a quenched local central limit theorem for continuous-time random walks in $\mathbb Z^d, d\ge 2$, in a uniformly-elliptic time-dependent balanced random environment which is ergodic under space-time shifts. We also obtain Gaussian upper and lower bounds for quenched and (positive and negative) moment estimates of the transition probabilities and asymptotics of the discrete Green function.
\end{abstract}

%% file: Sec191122.tex
\section{Introduction}~\label{sec:intro}
In this article we consider a random walk in a balanced uniformly-elliptic time-dependent random environment on $\Z^d, d\ge 2$. 

For $x,y\in\Z^d$, we write $x\sim y$ if $|x-y|_2=1$. 
Denote by $\mc P$ the set ({\it of  nearest-neighbor transition rates on $\Z^d$}) 
\[
\mc P:=\left\{v: \Z^d\times\Z^d\to[0,\infty)\bigg|v(x,y)=0 \text{ if }x\nsim y\right\}.
\]
Equip $\mc P$ with the the product topology and the corresponding Borel $\sigma$-field.
We denote by $\Omega\subset \mc P^{\R}$ the set of all measurable  functions
$\omega: t\mapsto \omega_t$ from $\R$ to $\mc P$
 and call every $\omega\in\Omega$ a time-dependent {\it environment}. 
For $\omega\in\Omega$, we define the parabolic difference operator 
\begin{align*}
\mc L_\omega u(x,t)
&=\sum_{y:y\sim x}\omega_t(x,y)(u(y,t)-u(x,t))+\partial_t u(x,t)
\end{align*}
for every bounded function $u:\Z^d\times\R\to\R$ which is differentiable in $t$. 
Let $(\hat X_t)_{t\ge 0}=(X_t,T_t)_{t\ge 0}$ denote the continuous-time Markov chain on $\Z^d\times\R$ with generator $\mc L_\omega$. Note that almost surely, $T_t=T_0+t$. We say that $(X_t)_{t\ge 0}$ is a {\it continuous-time random walk in the environment }$\omega$ and denote by $P_\omega^{x,t}$ its law (called the {\it quenched law})
with initial state $(x,t)\in\Z^d\times\R$. 

 We equip $\Omega\subset\mc P^\R$ with the induced product topology and let $\mb P$ be a probability measure on the Borel $\sigma$-field $\mc B(\Omega)$ of $\Omega$. 
An environment $\omega\in\Omega$ is said to be  {\it balanced} if  
\[
\sum_{y}\omega_t(x,y)(y-x)=0, \quad
\text{ for all } t\in\R, x\in\Z^d
\]
and {\it uniformly elliptic} if  there is a constant $\kappa\in(0,1)$ such that
\[
\kappa< \omega_t(x,y)<\tfrac{1}{\kappa} \quad \text{ for all } t\in\R, x,y\in\Z^d \text{ with } x\sim y.
\]
Let $\Omega_\kappa\subset\Omega$ \label{page:defomg} denote the set of  balanced and uniformly elliptic environments with ellipticity constant $\kappa\in(0,1)$. The measure $\mb P$ is said to be balanced and uniformly elliptic if
$\mb P\left(
\omega\in\Omega_\kappa
\right)=1$ for some $\kappa\in(0,1)$.


For each $(x,t)\in\Z^d\times\R$ we define the space-time shift $\theta_{x,t}\omega:\Omega\to\Omega$ by 
\[
(\theta_{x,t}\omega)_s(y,z):=\omega_{s+t}(y+x,z+x).
\]
We assume that the law $\mb P$ of the environment is translation-invariant and {\it ergodic} under the space-time shifts $\{\theta_{x,t}:x\in\Z^d, t\ge 0\}$. I.e, $P(A)\in\{0,1\}$ for any $A\in\mc B(\Omega)$ such that $\mb P(A\Delta \theta_{\hat x}^{-1}A)=0$ for all $\hat x\in\Z^d\times[0,\infty)$.

Given $\omega$, the environmental process
\begin{equation}\label{eq:def-omegabar}
\bar\omega_t:=\theta_{\hat X_t}\omega, \qquad t\ge 0,
\end{equation}
with initial state $\bar\omega_0=\omega$ is a Markov process on $\Omega$. With abuse of notation, we use $P_\omega^{0,0}$ to denote the quenched law of $(\bar\omega_t)_{t\ge 0}$.

\noindent{\bf Assumptions:} {\it throughout this paper, we assume that $\mb P$ is balanced, ergodic, and uniformly elliptic with ellipticity constant $\kappa>0$.}

We recall the quenched central limit theorem (QCLT) in \cite{DGR15}.

\begin{theorem}\cite[Theorem 1.2]{DGR15}\label{thm:recall} 
Under the above assumptions of $\mb P$, 
\begin{enumerate}[(a)\,]
\item 
there exists a unique invariant measure $\mb Q$ for the process $(\bar\omega_t)_{t\ge 0}$ such that $\mb Q\ll\mb P$ and $(\bar\omega_t)_{t\ge 0}$ is an ergodic flow under $\mb Q\times P_\omega^{0,0}$. 
\item\label{item:qclt}(QCLT)
 For $\mb P$-almost all $\omega$, $P_\omega^{0,0}$-almost surely, $(X_{n^2t}/n)_{t\ge 0}$ converges weakly, as $n\to\infty$, to a Brownian motion with deterministic non-degenerate covariance matrix $\Sigma={\rm diag}\{2E_{\mb Q}[\omega_0(0,e_i)], i=1,\ldots,d\}$.
\end{enumerate}

\end{theorem}
In the special case where the environment is time-independent, i.e, $\mb P(\omega_t=\omega_s \text{ for all } t,s\in\R)=1$, we say that the environment is {\it static}.
\begin{remark} 
For balanced random walks in a static, uniformly-elliptic, ergodic random environment on $\Z^d$, the QCLT has been first shown  by Lawler \cite{Lawl82}, which is a discrete version of the result of Papanicolaou and Varadhan \cite{PV82}. It is then generalized to static random environments with weaker ellipticity assumptions in \cite{GZ, BD14}. 
\end{remark}

\begin{remark}\label{rm3}
Write $\norm{f}_{L^p(\mb P)}:=(E_{\mb P}[|f|^p])^{1/p}$ for $p\in\R$.
Let 
\[
\rho(\omega):=\dd\mb Q/\dd\mb P.
\]
It is obtained in \cite{DGR15} that $\rho>0$,  $\mb P$-almost surely.
At the end of the proof of \cite[Theorem 1.2]{DGR15}, it is shown that $E_\mb Q[g]\le C\norm{g}_{L^{d+1}(\mb P)}$ for any bounded continuous function $g$, which  implies 
\begin{equation}\label{rho-moment}
E_{\mb P}[\rho^{(d+1)/d}]<\infty.
\end{equation}
Moreover, one of our main results (see	Theorem~\ref{thm:ap-property}\eqref{item:neg-moment} below) shows that there exists $q=q(\kappa,d)$ such that the $L^{-q}(\mathbb P)$ moment of $\rho$ is also bounded. 

For $(x,t)\in\Z^d\times\R$, set
\[
\rho_\omega(x,t):=\rho(\theta_{x,t}\omega).
\]
Since $\Omega$  is equipped with a product $\sigma$-field, for any fixed $\omega\in\Omega$, the map $\R\to\Omega$ defined by $t\mapsto \theta_{0,t}\omega$ is measurable. Hence for almost-all $\omega$, the function $\rho_\omega(x,t)$ is measurable in $t$.
Moreover, $\rho_\omega$ possesses the following properties.
For $\mb P$-almost all $\omega$,
\begin{enumerate}[(i)]
\item $\rho_\omega(x,t)\delta_x\dd t$ is an invariant measure for the process $\hat X_t$ under $P_\omega$;
\item $\rho_\omega(x,t)>0$ is the unique density (with respect to $\delta_x\dd t$) for an invariant measure of $\hat X$ that satisfies $E_\mb P[\rho_\omega(0,0)]=1$;
\item $\rho_\omega$ has a version which is absolutely continuous with respect to $t$ with 
\begin{equation}\label{rho-invariance}
\partial_t\rho_\omega(x,t)=\sum_{y} \rho_\omega(y,t)\omega_t(y,x)
\end{equation}
for almost every $t$, where  $\omega_t(x,x):=-\sum_{y:y\sim x}\omega_t(x,y)$.
\end{enumerate}
 The proof of these properties can be found in  \cite[Appendix]{arXiv}.
 
\end{remark}

As a main result of our paper, we will present the following {\it local limit theorem} (LLT), which is a finer characterization of the local behavior of the random walk than the QCLT. 
Let 
\[
\hat 0:=(0,0)\in\Z^d\times\R.
\] 
For $\hat x=(x,t),\hat y=(y,s)\in \Z^d\times\R$, $t\le s$, define 
\begin{equation}\label{eq:def-hk}
p^\omega(\hat x, \hat y):=P_\omega^{x,t}(X_{s-t}=y),
\quad
q^{\omega}(\hat x,\hat y)=\dfrac{p^\omega(\hat x,\hat y)}{\rho_\omega(\hat y)}.
\end{equation}

\begin{theorem}
[LLT]\label{thm:llt}
For $\mb P$-almost all $\omega$ and any $T>0$,
\[
\lim_{n\to\infty}\sup_{x\in\R^d,t>T}
\Abs{
n^dq^\omega(\hat 0;\floor{nx},n^2t)
-p_t^\Sigma(0,x)
}=0.
\]
Here $p_t^\Sigma(0,x)=[(2\pi t)^{d}\det\Sigma]^{-1/2}\exp(-x^T\Sigma^{-1} x/2t)$ is the transition kernel of the Brownian motion with covariance matrix $\Sigma$ and starting point $0$, and $\floor{x}:=(\floor{x_1},\ldots, \floor{x_d})\in\Z^d$ for $x\in\R^d$.
\end{theorem}

The proof of the LLT follows from Theorem~\ref{thm:recall} and a localization of the {\it heat kernel} $q^\omega(\hat 0,\cdot)$, an argument already implemented in \cite{Barlow-Hambly09} and \cite{ADS18} in the context of random conductance models.
For this purpose, the regularity of $\hat x\mapsto q^\omega(\hat 0,\hat x)$ is  essential. We use an analytical tool from classical PDE theory: the parabolic Harnack inequality (PHI) which yields not only H\"older continuity (cf. Corollary~\ref{cor:hoelder} below) of $q^\omega(\hat 0,\cdot)$ but also very sharp heat kernel estimates.
Note that for fixed $\hat x=(x,t)$, the function 
$u(\hat y)=q^\omega(\hat y,\hat x)$
satisfies
$\mc L_\omega u=0$ in $\Z^d\times(-\infty,t)$. 
However in our non-reversible model, we need to prove, instead of PHI for $\mc L_\omega$, the PHI (Theorem~\ref{thm-ah}) for the {\it adjoint operator} $\mc L^*_\omega$ (defined below), since the heat kernel
 $v_\omega(\hat x):=q^\omega(\hat 0,\hat x)$ 
 solves 
\begin{equation}\label{e27}
\mc L_\omega^*v(\hat x):=\sum_{y:y\sim x}\omega_t^*(x,y)(v(y,t)-v(\hat x))-\partial_t v(\hat x)=0
\end{equation}
for  $\hat x=(x,t)\in\Z^d\times(0,\infty)$, where 
\[
\omega_t^*(x,y):=\frac{\rho_\omega(y,t)\omega_t(y,x)}{\rho_\omega(x,t)}
\quad\mbox{ for }x\sim y\in\Z^d.
\]
Note that $\omega^*$ is not necessarily a balanced environment anymore.

For $r>0$, $x\in\R^d$, we let 
\[
B_r(x)=\{y\in\Z^d: |x-y|_2<r\}, \quad B_r=B_r(0)
\]
and define for $\hat x=(x,t)\in\R^d\times\R$ the {\it parabolic balls}
\begin{equation}\label{def:parab-ball}
Q_r(\hat x)=B_r(x)\times[t,t+r^2), \qquad Q_r=Q_r(\hat 0).
\end{equation}
Throughout this paper, unless otherwise specified, 
$C, c$ denote generic positive constants that depend only on $(d,\kappa)$, and which may differ from line to line. 
If $cB\le A\le CB$, we write 
\[
A\asymp B.
\]
 Our second main result is
\begin{theorem}[PHI for $\mc L^*_\omega$]\label{thm-ah}
For
$\mb P$-almost all $\omega$, 
any  non-negative solution $v$ of the adjoint equation
$\mc L_\omega^* v=0$ in $B_{2R}\times(0, 4R^2]
$ satisfies
\[
\sup_{B_R\times(R^2,2R^2)}v\le C\inf_{B_R\times (3R^2,4R^2]}v.
\]
\end{theorem}

In PDE, the Harnack inequality for the adjoint of non-divergence form elliptic differential operators was first 
proved by Bauman \cite{Baum84}, and was generalized to the parabolic setting by Escauriaza \cite{Esc00}. Our proof of Theorem~\ref{thm-ah} follows the main idea of \cite{Esc00}.

\begin{remark}
For time discrete  random walks in a static environment, Theorem~\ref{thm-ah} was obtained by Mustapha \cite{Mustapha06}. His argument follows basically \cite{Esc00}, and uses the PHI \cite[Theorem~4.4]{KT98} of Kuo and Trudinger in the time discrete situation.
 Moreover, in the static case, the volume-doubling property of the invariant distribution, which is the essential part of the proof of Theorem~\ref{thm-ah}, is much simpler, see \cite{FS84}. In our dynamical setting, a {\it parabolic} volume-doubling property (Theorem~\ref{thm:vd}) is required. To this end, we adapt ideas of  Safonov-Yuan \cite{SY} and results in the references therein \cite{FSY, Baum84,Garo} into our discrete space setting.
 \end{remark}

The main challenge in proving Theorem~\ref{thm-ah} is that $\mc L^*_\omega$ is neither balanced nor uniformly elliptic, and so the PHI for $\mc L_\omega$ (Theorem~\ref{Harnack}) is not immediately applicable. This is the main difference with  the random conductance model with symmetric jump rates
where 
\[
\omega_t(x,y)=\omega_t(y,x)=\omega^*_t(x,y),
\]
and thus which PHI for $\mc L_\omega$ is the same as PHI for $\mc L_\omega^*$. See  \cite{Andres14,Delm99,DD05,ACDS,HK16}.
%

%

Let us explain the main idea for the proof of Theorem~\ref{thm-ah}.
An important observation is that solutions $v$ of $\mc L_\omega^*$ can be expressed in terms of hitting probabilities of the {\it time-reversed} process, cf. Lemma~\ref{lem8} below. Thus to compare values of the adjoint solution, one only needs to estimate hitting probabilities of the {\it original} process that {\it starts from} the boundary. To this end, we will use a  ``boundary Harnack inequality" (Theorem~\ref{thm-bh}) which compares $\mc L_\omega$-harmonic functions near the boundary. We will also need a volume-doubling inequality for the invariant measure (Theorem~\ref{thm:vd}) to control the change of probabilities due to time-reversal.

Recall the heat kernel $q^\omega$ in \eqref{eq:def-hk}. For any $A\subset\R^d$ and $s\in\R$, let
\[\rho_\omega(A,s)=\sum_{x\in A\cap\Z^d}\rho_\omega(x,s).\]
  We write the  $\ell^2$-norm of $x\in\R^d$ as $|x|=|x|_2$. For $r\ge 0,t>0$, define
\begin{equation}\label{eq:def-function-h}
\mathfrak{h}(r,t)=\tfrac{r^2}{t\vee r}+r\log(\tfrac{r}{t}\vee 1).
\end{equation}
Note that $\mathfrak{h}(c_1r,c_2 t)\asymp \mathfrak{h}(r,t)$ for  constants $c_1,c_2>0$.

Our third main results are the following heat kernel estimates (HKE). \begin{theorem}[HKE]\label{thm:hke}
For $\mb P$-almost every $\omega$ and all $\hat x=(x,t)\in\Z^d\times(0,\infty)$,
\begin{equation}\label{eq:quenched-hke}
\frac{c}{\rho_\omega(B_{\sqrt t}(y),s)}
e^{-C\frac{|x|^2}{t}}
\le 
q^\omega(\hat 0,\hat x)
\le 
\frac{C}{\rho_\omega(B_{\sqrt t}(y),s)}
e^{-c\mathfrak{h}(|x|,t)}
\end{equation}
for all $s\in[0,t]$ and $y$ with $|y|\le |x|+c\sqrt t$. 
 Moreover, recalling the definition of $L^p(\mb P)$ in Remark~\ref{rm3}, there exists $p=p(d,\kappa)>0$  such that
\begin{align}
&\norm{P_\omega^{0,0}(X_t=x)}_{L^{(d+1)/d}(\mb P)}
\le \frac{C}{(t+1)^{d/2}}e^{-c\mathfrak{h}(|x|,t)}\label{eq:mm-hke1}\\
\text{and }\quad
&\norm{P_\omega^{0,0}(X_t=x)}_{L^{-p}(\mb P)}
\ge
\frac{c}{(t+1)^{d/2}}e^{-C\frac{|x|^2}{t}}\label{eq:mm-hke2}
\end{align}	
for all $(x,t)\in\Z^d\times(0,\infty)$. As a consequence,  setting 
$G^\omega(0,x)=\int_0^\infty P_\omega^{0,0}(X_t=x)\dd t$, we have for  $d\ge 3$ and $x\in\Z^d$, 
\begin{equation}\label{eq:mm-green}
\norm{G^\omega(0,x)}_{L^{(d+1)/d}(\mb P)}\asymp
\norm{G^\omega(0,x)}_{L^{-p}(\mb P)}\asymp
(|x|+1)^{2-d}.
\end{equation}
\end{theorem}

Note that for a general ergodic environment, the density $\rho_\omega$ does not have deterministic (positive) upper and lower bounds, thus one cannot expect deterministic quenched Gaussian bounds for $p^\omega(\hat 0,\hat x)$. However, our Theorem~\ref{thm:hke} shows that it has $L^{(d+1)/d}(\mb P)$ and $L^{-q}(\mb P)$ moment bounds. 
Furthermore, we can characterize asymptotics of the  Green function of the RWRE.
 Recall the notations $\Sigma$ (in Theorem~\ref{thm:recall} \eqref{item:qclt}), $p^\Sigma_t$, $\floor{x}$ (in Theorem~\ref{thm:llt}), and $\mathfrak{h}$ in \eqref{eq:def-function-h}.
\begin{corollary}\label{cor:q-estimates}
 The following statements are true for $\mb P$-almost every $\omega$.
\begin{enumerate}[(i)]
\item\label{cor:q-hke}  
There exists $t_0(\omega)>0$ such that 
for any $\hat x=(x,t)\in\Z^d\times(t_0,\infty)$,
\[
\frac{c}{t^{d/2}}e^{-\frac{C|x|^2}{t}}
\le q^\omega(\hat 0,\hat x)
\le 
\frac{C}{t^{d/2}}e^{-c\mathfrak{h}(|x|,t)}.
\]
As a consequence, the RWRE is recurrent when $d=2$ and transient when $d\ge 3$.
\item\label{cor:green1} 
When $d=2$,  for all $x\in \R^d\setminus\left\{0\right\}$,
\[\lim_{n\to\infty}\frac{1}{\log n}\int_0^\infty \left[q^\omega(\hat 0;0,t)-q^\omega(\hat 0;\floor{nx},t)\right]\dd t
=\frac{1}{\pi\sqrt{\det\Sigma}}.
\]
\item\label{cor:green2}
When $d\ge 3$, for all $x\in\R^d\setminus\left\{0\right\}$,
\[
\lim_{n\to\infty}n^{d-2}\int_0^\infty q^\omega(\hat 0;\floor{nx},t)\dd t
=\int_0^\infty p^\Sigma_t(0,x)\dd t.
\]
\end{enumerate}
\end{corollary}


Similar results as Corollary~\ref{cor:q-estimates}\eqref{cor:green1}\eqref{cor:green2} are also obtained recently for the conductance model \cite{ADS18}.

The organization of this paper is as follows. In Section~\ref{sec:vd}, we prove a parabolic volume-doubling property and an $A_p$ inequality for $\rho_\omega$, and obtain a new proof of the PHI for $\mc L_\omega$. In Section~\ref{sec:near-bdry}, we establish estimates of $\mc L_\omega$-harmonic functions near the boundary, showing both the interior elliptic-type and boundary PHI's. We prove the PHI for the adjoint operator (Theorem~\ref{thm-ah}) in Section~\ref{sec:proof-of-ah}. Finally, with the adjoint PHI, we prove Theorems~\ref{thm:llt}, \ref{thm:hke}, and Corollary~\ref{cor:q-estimates} in Section~\ref{sec:proof-llt-hke}. Section~\ref{sec:auxiliary-prob} contains probability estimates that are used in  previous sections. 
Some estimates and standard arguments can be found in the Appendix of the longer arXiv version \cite{arXiv} of this paper.

\section{A local volume-doubling property and its consequences}\label{sec:vd}

The purpose of this section is to obtain a parabolic volume-doubling property (Theorem~\ref{thm:vd}) and a negative moment estimate (Theorem~\ref{thm:ap-property}) for the density $\rho_\omega$.  The proof relies crucially on a volume-doubling property  for the hitting probabilities restricted in a finite ball (Theorem~\ref{thm:prob_doubling}), which is an improved version of \cite[Theorem 1.1]{SY} by Safonov and Yuan in the PDE setting.

 As a by-product, we obtain a new proof of the classical PHI of Krylov and Safonov \cite{KS80} in the lattice (Theorem~\ref{Harnack}). Our proof, which is of interest on its own, can be viewed as the parabolic version of Fabes and Stroock's \cite{FS84} proof in the elliptic PDE setting.
\subsection{Volume-doubling properties}
For a finite subgraph $D\subset\Z^d$, let
\[\label{page:bdry}
\partial D=\{y\in\Z^d\setminus D: y\sim x \mbox{ for some }x\in D\}, \quad
\bar D:=D\cup\partial D.
\]
For $\ms D\subset \Z^d\times\R$, define
the {\it parabolic boundary} of $\ms D$ as
\[
\partial^\p\ms D:=
\{
(x,t)\notin\ms D: \big(B_{1+\epsilon}(x)\times(t-\epsilon,t]\big)\cap\ms D\neq\emptyset
\mbox{ for all }\epsilon>0
\}.
\]
In the special case $\ms D=D\times[0,T)$ for some finite $D\subset\Z^d$,
it is easily seen that $\partial^\p\ms D=(\partial D\times[0,T])\cup(\bar D\times\{T\})$.
See figure~\ref{fig:p-bdry}.

By the optional stopping theorem, for any $(x,t)\in\ms D\subset\Z^d\times\R$ and any bounded integrable function $u$ on $\ms D\cup\partial^\p\ms D$,  

\begin{equation}\label{representation}
u(x,t)=-E_\omega^{x,t}\left[\int_0^\tau\mc L_\omega u(\hat X_r)\dd r\right]+
E^{x,t}_\omega[u(\hat X_\tau)],
\end{equation}
where $\tau=\inf\{r\ge 0: (X_r,T_r)\notin \ms D\}$. 

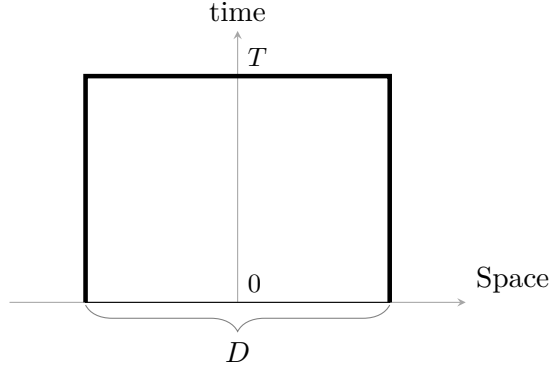
\begin{figure}[H]
\centering
\pgfmathsetmacro{\R}{2}
   \begin{tikzpicture}
    \draw[gray!70,->, >=stealth] (-1-\R,0)--(\R+1,0) node[black,above right] {Space};
    \draw[gray!70,->,>=stealth] (0,0) node[black, above right] {\small $0$}--(0,3.6) node[black,above] {time};
    \draw (-\R,0) rectangle (\R,3);
    \draw[line width=0.6mm] (-\R,0) to (-\R,3) to (\R,3) to (\R,0);
    \draw[gray,decorate,decoration={brace,amplitude=10pt,mirror},yshift=-1pt]
(-\R,0) -- (\R,0) node[black,below,midway, yshift=-10pt] {$D$};
    \node[above right] at (0,3) {\small $T$};
   \end{tikzpicture}
\caption{The parabolic boundary of $D\times[0,T)$.\label{fig:p-bdry}}   
\end{figure}

\begin{theorem}\label{thm:vd}
$\mb P$-almost surely, for every $r\ge 1/2$,
\[
\sup_{t:|t|\le r^2}\rho_\omega(B_{2r},t)\le C\rho_\omega(B_r,0).
\]
\end{theorem}
In PDE setting, this type of estimate was first established by Fabes and Stroock \cite{FS84} for adjoint solutions of non-divergence form elliptic operators, and then generalized by Escauriaza \cite{Esc00} to the parabolic case. 

 To obtain Theorem~\ref{thm:vd}, a crucial estimate is a volume-doubling property (Theorem~\ref{thm:prob_doubling}) for the hitting measure of the random walk, which we will prove by adapting some ideas of  Safonov and Yuan \cite[Theorem 1.1]{SY} in the PDE setting. Note that our proof of Theorem~\ref{thm:prob_doubling} relies on a probabilistic estimate (Lemma~\ref{lem1}) rather than the Harnack inequality (Theorem~\ref{Harnack}).
 
For any $A\subset\Z^d$, $s\in\R$, define the stopping time
\begin{equation}
\label{def:st-a-s}
\Delta(A,s)=\inf\{t\ge 0:\hat X_t\notin A\times(-\infty,s)\}.
\end{equation}

\begin{theorem}\label{thm:prob_doubling}
Assume $\omega\in\Omega_\kappa$. There exists $k_0=k_0(d,\kappa)$ such that for any $k\ge k_0$, $m\ge 2$, $r,s>0$ and any $y\in B_{k\sqrt s}$, we have
\[
P_\omega^{y,0}(X_{\Delta(B_{mk\sqrt s},s)}\in B_{2r})\le C_k P_\omega^{y,0}(X_{\Delta(B_{mk\sqrt s}, s)}\in B_r).
\]
Here $C_k$ depends only on $(k,d,\kappa)$. In particular, for any $k\ge 1$, $|y|\le k\sqrt s$,
\[
P_\omega^{y,0}(X_s\in B_{2r})\le C_k P_\omega^{y,0}(X_s\in B_r).
\]
\end{theorem}

\begin{proof}[Proof of Theorem~\ref{thm:prob_doubling}:]
Since $B_1=\{0\}$, we only consider $r\ge 1/2$. Fix $s$, $r$. For $\rho\ge 0$, $k\ge k_0$, define $L_{k,\rho}=B_{k\rho}\times\{s-\rho^2\}$ and 
\[
D_{k,\rho}=\bigcup_{R\le\rho}L_{k,R}=\{(x,t)\in\Z^d\times(-\infty,s]: |x|/k\le \sqrt{s-t}\le\rho\}.
\]
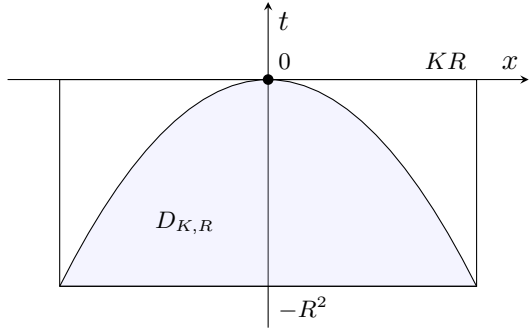
\begin{figure}[h]
  \centering
  \begin{tikzpicture}
  [
  declare function={
    f(\t) = -\t^2/(\K*\K);
  }
]
\pgfmathsetmacro{\R}{2}
\pgfmathsetmacro{\K}{2}
\pgfmathsetmacro{\Rho0}{0.8*\R}
\pgfmathsetmacro{\Rho}{0.9*\R}
    \begin{axis}[
        xmin=-\K*\R-1, xmax=\K*\R+1,
        ymin=-\R*\R-0.8, ymax=1.5,
        ticks=none, 
              xlabel=$x$,
            ylabel=$t$,
        axis lines=middle,
        unit vector ratio=1 1
      ]  
      \addplot[samples=30,domain=-\K*\R:\K*\R, name path=A] {f(x)};
      \addplot[domain=-\K*\R:\K*\R,name path=D](x,-\R*\R) node[pos=0.5, below right]{\footnotesize $s-\rho^2$};
       \addplot[fill=blue!20, fill opacity=0.2] fill between [of =A and D];
      \draw (-\K*\R,0) rectangle (\K*\R,-\R*\R);   
      \node[above left] at (\K*\R,0) {\footnotesize $k\rho$};
 \fill (0,0) coordinate (o) circle (2pt);
\node at (o)[above right] {\footnotesize $(0,s)$};
\node at (60:-0.8*\R*\R) {\footnotesize $D_{k,\rho}(s)$};
    \end{axis}
  \end{tikzpicture}
\caption{The shaded region is $D_{k,\rho}$.\label{fig:dkr}}  
\end{figure}
Note that $k_0>6$ is a large constant to be determined.
For any $R\le\rho$, by Lemma~\ref{lem1} below, there exists $\alpha_k>0$ depending on $(k,\kappa,d)$ such that
\begin{align*}
\min_{(x,t)\in L_{k,R}}P^{x,t}_\omega(X_{\Delta(B_{mk\rho},s)}\in B_r)
&\ge (\tfrac{r}{2kR}\wedge\tfrac{1	}{2})^{\alpha_k}.
\end{align*}	
Let $\beta_k>1$ be a large constant to be determined later. Then, letting
\[
v_\rho(\hat x)=(8k)^{\alpha_k}(\beta_k+1)P_\omega^{\hat x}(X_{\Delta(B_{mk\rho},s)}\in B_r)-P_\omega^{\hat x}(X_{\Delta(B_{mk\rho},s)}\in B_{2r}),
\]
we get $\inf_{D_{k,R}}v_\rho\ge (\beta_k+1)(\tfrac{4r}{R}\wedge2k)^{\alpha_k}-1$ for $0<R\le \rho$.  In particular, 
$\inf_{D_{k,(4r)\wedge\rho}}v_\rho\ge \beta_k$ for $\rho\ge 0$. 
Set 
\[
R_\rho=\sup\{R\in[0,m\rho]:\inf_{D_{k,R}}v_\rho\ge 0\}.
\]
Clearly, $R_\rho\ge (4r)\wedge\rho$.  We will prove that 
\begin{equation}\label{eq:R-rho}
R_\rho\ge \rho \text{ for all }\rho>0.
\end{equation}

Assuming \eqref{eq:R-rho} fails, then $R_\rho<\rho$ for some $\rho>4r$. We will show that this is impossible via contradiction.
First, for such $\rho>4r$, we claim that there exists a constant $\gamma=\gamma(d,\kappa)>0$ such that
\begin{equation}
\label{eq:lowb_vrho}
\min_{L_{1,R}}v_\rho\ge \beta_k(\frac{r}{R})^\gamma  
\quad\text{ for all }R\in[2r,R_\rho).
\end{equation}
By Lemma~\ref{lem1}, 
$g(R, R/2):=
\min_{x\in B_R}P_\omega^{x,s-R^2}(X_{\Delta(B_{2R},s-(R/2)^2)}\in B_{R/2})\ge C$. 
Further, by the Markov property and that $R_\rho<\rho$, for $R\in[2r,R_\rho)$, 
\begin{align*}
&\min_{x\in B_R}P_\omega^{x,s-R^2}(X_{\Delta(D_{k,m\rho},s-(R/2^n)^2)}\in B_{R/2^n})\\
&\ge
g(R,R/2)\cdots g(R/2^{n-1},R/2^n)\ge C^n.
\end{align*}	
Since $v_\rho(\hat X_t)$ is a martingale in $B_{mk\rho}\times(-\infty,s)$ and that $v_\rho\ge 0$ in $D_{k,R_\rho}$, 
choosing $n$ such that $R/2^n\le r<R/2^{n-1}$, the above inequality yields
\begin{align*}
v_\rho(x,s-R^2)
&\ge P_\omega^{x,s-R^2}(X_{\Delta(D_{k,m\rho},s-(R/2^n)^2)}\in B_{R/2^n})\inf_{D_{k,r}}v_\rho\\
&\ge C^n\beta_k\ge \beta_k(\frac{r}{R})^c
\end{align*}	
for $R\in[2r,R_\rho)$ and $x\in B_R$.  Display \eqref{eq:lowb_vrho} is proved.

Next, we will show that for $R\in[2r,R_\rho)$, 
\begin{equation}
\label{eq:upperb-frho}
f_\rho(R):=\sup_{\partial B_{kR}\times[s-R^2,s)}v_\rho^-\le (\frac{r}{R})^{-c/\log q_k},
\end{equation}
where $v_\rho^-=\max\{0,-v_\rho\}$, $q_k=1-\tfrac{c_1}{k}$ and $c_1>0$ is a constant to be determined. 
Noting that $v_\rho^-=0$ in $D_{k,R_\rho}\cup (B_{2r}^c\times\{s\})$ and that $v_\rho^-(\hat X_t)$ is a sub-martingale, we know that $f_\rho(R)$ is a decreasing function for $R\in(\frac{2r}{k},R_\rho)$.
 Further,
for any $(x,t)\in \partial B_{kR}\times[s-R^2,s)$ with $q_k R\in(\frac{2r}{k},R_\rho)$, by the optional stopping lemma, in view of Theorem~\ref{thm:fluctuation} below,
\begin{align*}
v_\rho^-(x,t)
&\le 
P_\omega^{x,t}(X_a\in\partial B_{kq_kR}\text{ for some }a\in[0,R^2])f(q_kR)\\
&\le 
P_\omega^{x,t}\left(\sup_{0\le a\le R^2}|X_a-x|\ge (1-q_k)kR \right)f(q_kR)\\
&\stackrel{Theorem~\ref{thm:fluctuation}}{\le} C(e^{-cc_1^2}+e^{-c_1 R})f(q_kR)
\le f(q_kR)/2,
\end{align*}
if $c_1$ is chosen to be big enough.
So $f_\rho(R)\le f_\rho(q_kR)/2$. Let $n\ge 0$ be the integer such that $q_k^{n+1}R<r\le q_k^n R$. We conclude that for $R\in[2r,R_\rho)$,
\[
f_\rho(R)\le 2^{-n}f_\rho(q_k^nR)\le \left(
\frac{r}{R}
\right)^{-c/\log q_k}f_\rho(r).
\]
Inequality \eqref{eq:upperb-frho} then follows from the fact $v_\rho^-\le 1$.

Finally, if $R_\rho<\rho$ for some $\rho>4r$,
 let $\tau=\inf\{t\ge0:\hat X_t\notin (B_{mk\rho}\times(-\infty,s))\setminus(\bar B_{kR/2}\times[s-(R/2)^2,s))\}$. Since $v_\rho=0$ on $\partial^\p (B_{mk\rho}\times(-\infty,s))\setminus(B_{2r}\times\{s\})$, by the optional stopping lemma, for $R\in[R_\rho,2R_\rho)$ and  $x\in B_{kR}$, 
\begin{align*}
&v_\rho(x,s-R^2)\\
  &=E_\omega^{x,s-R^2}[v_\rho(\hat X_\tau)\mathbbm{1}_{\hat X_\tau\in B_{kR/2}\times\{s-(R/2)^2\} \text{ or }\hat X_\tau\in\partial B_{kR/2}\times[s-(R/2)^2,s)}]\\
  &\ge
  P_\omega^{x,s-R^2}(X_{\Delta(B_{2k\rho},s-(R/2)^2)}\in B_{R/2})\min_{L_{1,R/2}}v_\rho-f_\rho(R/2)\\
  &\stackrel{Lemma~\ref{lem1},\eqref{eq:lowb_vrho},\eqref{eq:upperb-frho}}{\ge} 
  A_k\beta_k(\frac{2r}{R})^\gamma-(\frac{2r}{R})^{-c/\log q_k},
  \end{align*}	  
where $A_k$ depends on $(k,\kappa,d)$. Taking $k_0>c_1$ to be big enough such that $-c/\log q_k>\gamma$ for $k\ge k_0$ and choosing $\beta_k>A_k^{-1}$, the above inequality then implies $\inf_{D_{k,2R_\rho}}v\ge 0$, which contradicts our definition of $R_\rho$. Display \eqref{eq:R-rho} is proved, and therefore, $\min_{x\in B_{k\sqrt s}}v_{\sqrt s}(x,0)\ge 0$. The theorem follows.
\end{proof}

\begin{corollary}
\label{cor:space-time-vd}
Let $\omega\in\Omega_\kappa$ and $k_0$ as in Theorem~\ref{thm:prob_doubling}. For any $r>0, k\ge k_0, m\ge 2, s>0$ and $y\in B_{k\sqrt s}$, we have
\[
\sup_{t\ge 0: |t-s|\le r^2}P_\omega^{y,0}(X_{\Delta(B_{mk\sqrt s},t)}\in B_{2r})\le C_k P_\omega^{y,0}(X_{\Delta(B_{mk\sqrt s},s)}\in B_r),
\]
where $C_k$ depends on $(k,\kappa,d)$. In particular, for any $k\ge 1, |y|\le k\sqrt s$,
\[
\sup_{t\ge 0: |t-s|\le r^2}P_\omega^{y,0}(X_t\in B_{2r})\le C_k P_\omega^{y,0}(X_s\in B_r).
\]
\end{corollary}
\begin{proof}
It suffices to consider $r<\sqrt s$, because otherwise, by Lemma~\ref{lem1}, the right side is bigger than a constant.  
When $t\in[0\vee(s-r^2),s]$,
\begin{align*}
&\min_{y\in B_{k\sqrt s}}P_\omega^{y,0}(X_{\Delta(B_{mk\sqrt s},s)}\in B_{4r})\\
&\ge 
\min_{y\in B_{k\sqrt s}}P_\omega^{y,0}(X_{\Delta(B_{mk\sqrt s},t)}\in B_{2r})\min_{x\in B_{2r}}P_\omega^{x,t}(X_{s-t}\in B_{2r}(x))\\
&\stackrel{Lemma~\ref{lem1}}{\ge} 
C \min_{y\in B_{k\sqrt s}}P_\omega^{y,0}(X_{\Delta(B_{mk\sqrt s},t)}\in B_{2r}).
\end{align*}
By Theorem~\ref{thm:prob_doubling}, we can replace $4r$ in the above inequality by $r$.

When $t\in[s,s+r^2]$, for any $y\in B_{k\sqrt s}$,
\begin{align*}
&P_\omega^{y,0}(X_{\Delta(B_{mk\sqrt s},t)}\in B_{2r})\\
&\le\sum_{n=0}^\infty\sum_{x:|x|\in[2^nr,2^{n+1}r)}
P_\omega^{y,0}(X_{\Delta(B_{mk\sqrt s},s)}=x) P_\omega^{x,s}(X_{t-s}\in B_{2r})
\\
&\stackrel{Corollary~\ref{cor:prob-upperb}}{\le }
C\sum_{n=0}^\infty P_\omega^{y,0}(X_{\Delta(B_{mk\sqrt s},s)}\in B_{2^n r})(e^{-c2^nr}+ e^{-c 4^n}).
\end{align*}
Observing that (cf. Theorem~\ref{thm:prob_doubling}) 
\[
P_\omega^{y,0}(X_{\Delta(B_{mk\sqrt s},s)}\in B_{2^n r})
\le C^n P_\omega^{y,0}(X_{\Delta(B_{mk\sqrt s},s)}\in B_{r}),
\]
our proof is complete.
\end{proof}

\begin{proof}[Proof of Theorem~\ref{thm:vd}:]
Let $k_0\ge 2$ be as in Theorem~\ref{thm:prob_doubling}. Recall 
$\bar\omega_t, Q_r, \Delta$ in \eqref{eq:def-omegabar}, \eqref{def:parab-ball}, \eqref{def:st-a-s}.
For fixed $\xi\in\Omega_\kappa$, define a probability measure $\Q_R=\Q_R^\xi$ on $\{\theta_{\hat x}\xi: \hat x\in Q_R\}$ such that for any bounded measurable $f\in\R^{\Omega}$, 
\[
E_{\Q_R}[f]=\frac{1}{C_R}E_\xi^{0,-R^2}
[\int_0^{\Delta(B_{2k_0 R},R^2)}f(\bar\xi_s)\mathbbm{1}_{\hat X_s\in Q_R}\dd s],
\]
where $C_R$ is a renormalization constant such that $\Q_R$ is a probability. 

First, we claim that $C_R\asymp R^2$. Clearly, $C_R\le 2R^2$. On the other hand,
\begin{align*}
C_R &= E_\xi^{0,-R^2}[\int_0^{\Delta(B_{2k_0 R},R^2)}\mathbbm{1}_{\hat X_s\in Q_R}\dd s]\\
&\ge P_\xi^{0,-R^2}(X_{\Delta(B_R,0)}\in B_{R/2})\min_{x\in B_{R/2}}E_\xi^{x,0}[\Delta(B_R,R^2)]\\
&\stackrel{Lemma~\ref{lem1}}{\ge} C\min_{x\in B_{R/2}}E_\xi^{x,0}[\Delta(B_R,R^2)].
\end{align*}
Since $|X_t-X_0|^2-\tfrac d\kappa t$ is a supermartingale, denoting $\tau=\Delta(B_R,R^2)$,  we have
$0\ge E_\xi^{x,0}[|X_\tau-x|^2-\tfrac d\kappa \tau]$. Hence for any $x\in B_{R/2}$, 
\[
E_\xi^{x,0}[\tau]\ge cE_\xi^{x,0}[|X_\tau-x|^2]
\ge CR^2 P_\xi^{x,0}(\tau<R^2),
\]
which implies $E_\xi^{x,0}[\tau]\ge cR^2$. Thus $C_R\ge CR^2$ and so $C_R\asymp R^2$.

Next, since $\Omega$ is pre-compact, by Prohorov's theorem, there is a subsequence of $\Q_R$ that converges weakly, as $R\to\infty$, to a probability measure $\tilde\Q$ on $\Omega$. We will show that $\tilde \Q$ is an invariant measure of the process $(\bar\omega_t)$. Indeed, let $p_R=p_{R,\xi}$ denote the kernel $p_R(\hat x;y,s):=P_\xi^{\hat x}(\hat X_{\Delta(B_{2k_0 R}, s)}=(y,s))$. 
Then, letting $\ms Lf(\omega)=\sum_{e}\omega_0(0,e)[f(\theta_{e,0}\omega)-f(\omega)]+\partial_t f(\theta_{0,t}\omega)|_{t=0}$ denote the generator of the process $(\bar\omega_t)$, and $\hat y:=(y,s)$, we have
\begin{equation}\label{eq: QR}
E_{\Q_R}[\ms L f(\omega)]
=C_R^{-1}\sum_{y\in B_R}\int_0^{R^2}p_R(0,-R^2;\hat y)\ms Lf(\theta_{\hat y}\xi)\dd s
\end{equation}
for $f\in\dom(\ms L)$, where $\dom(\ms L)$ denotes the domain of the generator $\ms L$.
Note that similar to $\rho_\omega$, the function  $v(\hat x)=p_R(0,-R^2;\hat x)$ satisfies the equality \eqref{rho-invariance}: $\ms L^Tv(\hat x)=0$ for $\hat x\in B_{2R}\times(-R^2,R^2)$, where $\ms L^T v(\hat x)=\sum_y v(y,t)\omega_t(y,x)-\partial_t v(x,t)$. Hence, using integration by parts,  
\begin{align}\label{eq:qr-difference}
&\Abs{\sum_{y\in B_R}\int_0^{R^2}p_R(0,-R^2;\hat y)\ms Lf(\theta_{\hat y}\xi)\dd s}
\nn\\&\le 
C\norm{f}_\infty\int_0^{R^2}\sum_{y\in \bar B_R\setminus\mathring{B}_R}p_R(0,-R^2;\hat y)\dd s+2\norm{f}_\infty
\end{align}
for all $f\in\dom(\ms L)$, 
where $\overset{\circ}{B}_R=\{x\in B_R: x\not\sim\partial B_R\}$. Observe that 
\[
u(\hat x)=\int_0^{R^2}\sum_{y\in \bar B_R\setminus\mathring{B}_R}p_R(\hat x;\hat y)\dd s=E^{\hat x}_\xi[\int_0^{\Delta(B_{2k_0 R},R^2)}\mathbbm{1}_{\hat X_t\in \bar B_R\setminus\mathring{B}_R\times(0,R^2)}]\dd t
\]
satisfies $\mc L_\xi u(\hat x)=-\mathbbm{1}_{\hat x\in \bar B_R\setminus\mathring{B}_R\times[0,R^2)}$ for $\hat x\in\ms D:= B_{2k_0R}\times[-R^2,R^2)$ and $u|_{\partial^\p\ms D}=0$. By the parabolic maximum principle \cite[Theroem A.3.1]{arXiv}, we get
$u(0,-R^2)\le CR^{(2d+1)/(d+1)}$.  Hence, by \eqref{eq: QR}, \eqref{eq:qr-difference}, and $C_R\asymp R^2$, 
\[
\lim_{R\to\infty}E_{\Q_R}[\ms Lf]=0 \qquad
\forall \text{ bounded function }f\in\dom(\ms L),
\]
and so $E_{\tilde\Q}[\ms L f]=0$, which implies that $\tilde\Q$ is an invariant measure of $(\bar\omega_t)$. 

Furthermore, we will show that $\tilde\Q\ll\mb P$. Notice that the function
\[
w(\hat x):=E_\xi^{\hat x}
[\int_0^{\Delta(B_{2k_0 R},R^2)}f(\bar\xi_s)\mathbbm{1}_{\hat X_s\in Q_R}\dd s]
\]
satisfies $\mc L_\xi w(\hat x)=-f(\theta_{\hat x}\xi)\mathbbm{1}_{\hat x\in Q_R}$ in $\ms D$ and $w|_{\partial^\p\ms D}=0$. By \cite[Theorem A.3.1]{arXiv}, for any bounded measurable $f\in\R^\Omega$,
\[
E_{\Q_R}[f]\le CR^{-2}w(0,-R^2)
\le 
C\left[\int_0^{R^2}\sum_{x\in B_R}|f(\theta_{x,t}\xi)|^{d+1}\dd t\bigg/(R^{2+d})
\right]^{1/(d+1)},
\]
which, by the multi-dimensional ergodic theorem, yields $E_{\tilde\Q}[f]\le C\norm{f}_{L^{d+1}(\mb P)}$ as we take $R\to\infty$. So $\tilde \Q\ll\mb P$. By Theorem~\ref{thm:recall}, $\tilde \Q=\Q$. 

Finally, since $\Q_R\Rightarrow\Q$, for any bounded measurable $f\in\R^\Omega$, 
\begin{align}\label{eq:formula-rho-br}
&E_{\mb P}[\rho_\omega(B_r,t)f]=\sum_{x\in B_r}E_{\Q}[f(\theta_{x,-t}\omega)]\\
&=\lim_{R\to\infty}\sum_{x\in B_r,y\in B_R}\int_0^{R^2}
P_\xi^{0,-R^2}(X_{\Delta(B_{2k_0 R},s)}=y)f(\theta_{x+y,s-t}\xi)\dd s/C_R.\nn
\end{align}	
Hence, for  any  measurable function $f\ge 0$, $|t|\le r^2$, and $\mb P$-a.a. $\xi$,
\begin{align*}
&E_{\mb P}[\rho_\omega(B_r,0)f]\\
&\ge 
\lim_{R\to\infty}\sum_{z\in B_{R-r}}\int_0^{R^2}
P_\xi^{0,-R^2}(X_{\Delta(B_{2k_0 R},s)}\in B_r(z))f(\theta_{z,s}\xi)\dd s/C_R\\
&\stackrel{Corollary~\ref{cor:space-time-vd}}{\ge} 
C\lim_{R\to\infty}\sum_{z\in B_{R-r}}\int_0^{R^2}
P_\xi^{0,-R^2}(X_{\Delta(B_{2k_0 R},s+t)}\in B_{2r}(z))f(\theta_{z,s}\xi)\dd s/C_R\\
&\ge
C \lim_{R\to\infty}\sum_{x\in B_{2r},y\in B_{R-3r}}\int_{r^2}^{R^2-r^2}
P_\xi^{0,-R^2}(X_{\Delta(B_{2k_0 R},s)}=y)f(\theta_{x+y,s-t}\xi)\dd s/C_R\\
&\stackrel{\eqref{eq:formula-rho-br}}{=}CE_{\mb P}[\rho_\omega(B_{2r},t)f].
\end{align*}	
Since $f$ is arbitrary, the theorem follows.
\end{proof}

\begin{remark}
By Theorem~\ref{thm:vd}, for any $r\ge 1$,
\begin{equation}\label{eq:rho-asymp}
\frac{c}{r^2}\int_0^{r^2} \rho_\omega(B_r, s)\dd s
\le 
\rho_\omega(B_r,0)
\le 
\frac{C}{r^2}\int_0^{r^2} \rho_\omega(B_r, s)\dd s.
\end{equation}
Hence, by the multi-dimensional ergodic theorem, for $\mb P$-almost every $\omega$,
\begin{equation}\label{eq:rho-ergodic}
c\le \varliminf_{r\to\infty}\frac{1}{|B_r|}\rho_\omega(B_r,0)
\le
\varlimsup_{r\to\infty}\frac{1}{|B_r|}\rho_\omega(B_r,0)
\le C.
\end{equation}
\end{remark}

\subsection{$A_p$ property and a new proof of the PHI for $\mc L_\omega$}\label{sec:ap&harnack}
We endow $\Z^d$ with the discrete topology and counting measure, and equip $\Z^d\times\R$ with the corresponding product topology and measure (where $\R$ has the usual topology and measure).  For $\ms D\subset\Z^d\times\R$, let $|\ms D|$ be its measure, and denote
the integration over $\ms D$  by $\int_{\ms D}f$ . For instance,  
\begin{equation}\label{eq:def-integral}
\int_{B_R\times[0,T]}f=\sum_{x\in B_R}\int_0^T f(x,t)\dd t,
\end{equation}
and $|\ms D|=\int_{\ms D}1$. 
For $p>0$, we define a norm 
\begin{equation}
\label{eq:def-norm}
\norm{f}_{\ms D, p}:=(\int_{\ms D}|f|^p/|\ms D|)^{1/p}.
\end{equation}

A function $v\in\R^{\Z^d\times\R}$ is called an {\it adjoint solution} of $\mc L_\omega$ in $\ms D=B_R\times[T_1,T_2)$ if 
$\int_{\ms D}v\mc L_\omega\phi=0$ 
for any test function $\phi(x,t)\in \R^{\Z^d\times\R}$ that is supported on $B_{R}\times(T_1,T_2)$ and smooth in $t$.

For any function $w$ defined on $E\subset\Z^d\times\R$, we write
$w(E):=\int_E w$.

\begin{lemma}\label{lem:pre-RH}
Recall $\norm{\cdot}_{\ms D,p}$ in \eqref{eq:def-norm} and the parabolic balls $Q_r$ in \eqref{def:parab-ball}. 
Let $\omega\in\Omega_\kappa$. For any non-negative adjoint solution $v$ of $\mc L_\omega$ in $Q_{2r}$, $r>10$,
\[
\norm{v}_{Q_r,(d+1)/d}\le C\norm{v}_{Q_{3r/2},1}.
\]
\end{lemma}
\begin{proof}
Denote the continuous balls of radius $r$ by
\begin{equation}\label{eq:def-continuousball}
\mc O_r=\{x\in\R^d:|x|_2<r\}
\quad\text{ and }\quad
\mc O_r(y)=y+\mc O_r, \quad y\in\R^d.
\end{equation}
Let $\phi_0\ge 0$ be a smooth (with respect to $t$) function supported on  $\mc O_{3/2}\times[0,9/4)$ with ${\phi_0}|_{\mc O_1\times[0,1)}=1$ and set $\phi(x,t)=\phi_0(x/r,t/r^2)$. Let $f$ any non-negative smooth function supported on $Q_{r}$ with $\norm{f}_{Q_r,d+1}=1$ and let $u\in[0,\infty)^{\Z^d\times\R}$ be supported on $Q_{9r/5}$ with $L_\omega u=-f$ in $Q_{9r/5}$. Since
\begin{align*}
0&=\int v\mc L_\omega(\phi u)=\int v\phi\mc L_\omega u+\int v u\mc L_\omega\phi+\sum_{x,y}\int_\R v(x,t)\omega_t(x,y)\nabla_{x,y}u\nabla_{x,y}\phi\dd t,
\end{align*}	
where  $\nabla_{x,y}u(\cdot,t):=u(x,t)-u(y,t)$ and (cf. \eqref{eq:def-integral}) $\int=\int_{\Z^d\times\R}$, we get
\[
\int v\phi f=\int v u\mc L_\omega\phi+\sum_{x,y}\int_\R v(x,t)\omega_t(x,y)\nabla_{x,y}u\nabla_{x,y}\phi\dd t=:\rom{1}+\rom{2}.
\] 
By the maximum principle (\cite[Theorem~A.3.1]{arXiv}), $u\le Cr^2\norm{f}_{Q_r,d+1}\le Cr^2$. Thus, using $|\mc L_\omega\phi|\le C/r^2$, we get $|\rom{1}|\le Cv(Q_{3r/2})$.
Further, noting that
\begin{align*}
|\sum_{x,y}\int_\R v(x,t)\omega_t(x,y)(\nabla_{x,y}u)^2\dd t|
&=|\int v\mc L_\omega (u^2)-2\int vu\mc L_\omega u|\\
&=2|\int vu\mc L_\omega u|
\le Cr^2\int vf,
\end{align*}	
we have
\begin{align*}
|\rom{2}|&\le \left(\sum_{x,y}\int_\R v(x,t)\omega_t(x,y)(\nabla_{x,y}\phi)^2\dd t\right)^{1/2}
\left(\sum_{x,y}\int_\R v(x,t)\omega_t(x,y)(\nabla_{x,y}u)^2\dd t\right)^{1/2}\\
&\le 
Cv(Q_{3r/2})^{1/2}(\int vf)^{1/2}.
\end{align*}	
Hence we obtain
$\int v f\le \int v\phi f\le Cv(Q_{3r/2})+Cv(Q_{3r/2})^{1/2}(\int vf)^{1/2}$ and so $v(Q_{3r/2})\ge c\int vf$. The lemma follows by taking supremum over all $f$ with $\norm{f}_{Q_r,d+1}=1$.
\end{proof}

For $\hat x=(x_1,\ldots,x_d, t)$, define {\it parabolic cubes} with side-length $r>0$ as
\begin{equation}\label{eq:def-parakub}
K_r(\hat x)=(\prod_{i=1}^d[x_i-r,x_i+r)\cap\Z^d)\times[t,t+r^2), 
\quad
K_r=K_r(\hat 0).
\end{equation}
We say that a function $w\in\R^{\Z^d\times\R}$ satisfies the {\it reverse H\"older inequality} $RH_q(\ms D)$, $1<q<\infty$, if for any parabolic subcube $K$ of $\ms D$,
\[\tag{$RH_q$}
\norm{w}_{K,q}\le C\norm{w}_{K,1}.
\]
We say that $w$ belongs to the {\it $A_p(\ms D)$ class} (with  $A_p$ bound  $A$), $1<p<\infty$, if there exists $A<\infty$  such that, for any parabolic subcube $K$ of $\ms D$,
\[\tag{$A_p$}
\norm{w}_{K, 1}\norm{1/w}_{K, 1/(p-1)}\le A 
\]
Recall the stopping time $\Delta$ in \eqref{def:st-a-s}. For $R>0$, $\hat y\in B_{2R}\times\R$, let 
\begin{equation}\label{eq:def-gr}
g_R(\hat y;x,t)=P_\omega^{\hat y}(X_{\Delta(B_{2R},t)}=x).
\end{equation}

\begin{corollary}\label{cor:reverse-holder}
Let $\omega\in\Omega_\kappa, R>0$. Recall $k_0$ in Theorem~\ref{thm:prob_doubling}. 
\begin{enumerate}[(i)]
\item $\rho_\omega$ satisfies $RH_{(d+1)/d}(\Z^d\times\R)$.
\item 
For any $y\in B_R$, $v_y(\hat x)=g_R(y,0;\hat x)$ satisfies 
\[RH_{(d+1)/d}(B_{R/2}\times[R^2/(2k_0^2),R^2/k_0^2]).\]
\end{enumerate}
\end{corollary}
\begin{proof}
Note that $\rho_\omega$, $v_y$ are adjoint solutions with volume-doubling properties Theorem~\ref{thm:vd} and Corollary~\ref{cor:space-time-vd}. The corollary follows from Lemma~\ref{lem:pre-RH}.
\end{proof}
 
\begin{theorem}
\label{thm:ap-property}
Let $\omega, R, k_0, v_y$ be the same as in Corollary~\ref{cor:reverse-holder}. There exist $p=p(d,\kappa)>1, A=A(d,\kappa)$ such that, for $\mb P$-a.e. $\omega$, 
\begin{enumerate}[(a)]
 \item\label{item:neg-moment} $\rho_\omega\in A_p(\Z^d\times\R)$ with $A_p$ bound $A$. As a consequence,
 \[E_{\mb P}[\rho_\omega^{-1/(p-1)}]<\infty;\]
 \item For any $y\in B_R$, $v_y$ belongs to $A_p(B_{R/2}\times[R^2/(2k_0^2), R^2/k_0^2])$ with $A_p$ bound $A$. As a consequence, for any $E\subset K$ where $K$ is a parabolic subcube of $B_{R/2}\times[R^2/(2k_0^2), R^2/k_0^2]$,
 \begin{equation}\label{eq:A-infty}
\frac{g_R(y,0;E)}{g_R(y,0;K)}\ge C\left(\frac{|E|}{|K|}\right)^c.
 \end{equation}
 \end{enumerate} 
\end{theorem}

\begin{proof}
 See  \cite[Section~A.6]{arXiv}.
\end{proof}

\begin{remark}
The fact that $(RH)$ implies $(A_p)$ is a classical result in harmonic analysis. See e.g, \cite[pg.246-249]{CF74}, \cite[pg. 213-214]{Stein}.
In the elliptic non-divergence form PDE setting, the $A_p$ inequality for adjoint solutions was proved by Bauman \cite{Baum84}, and estimate of form \eqref{eq:A-infty} was used by Fabes and Stroock \cite{FS84} to obtain a short proof of the elliptic Harnack inequality. 
\end{remark}
In what follows, using \eqref{eq:A-infty}, we will obtain a new proof of the parabolic Harnack inequality (Theorem~\ref{Harnack}). Our proof follows the ideas of \cite{FS84}. Note that 
in our parabolic setting, the local volume-doubling property (Corollary~\ref{cor:space-time-vd}) played a crucial role in the proof of \eqref{eq:A-infty}.
\begin{theorem}[PHI for $\mc L_\omega$]\label{Harnack}
Assume $\omega\in\Omega_\kappa$ and $\theta>0$.
Let $u$ be a non-negative function that satisfies $\mc L_\omega u= 0$ in $B_R\times(0, \theta R^2)$. Then, for $0<\theta_1<\theta_2<\theta_3<\theta$, there exists a constant $C=C(\kappa,d,\theta_1,\theta_2,\theta_3,\theta)$ such that
\[
\sup_{B_{R/2}\times(\theta_2 R^2,\theta_3 R^2)}u\le C\inf_{B_{R/2}\times[0, \theta_1R^2)}u.\tag{PHI}
\]
\end{theorem}
We remark that in discrete time setting, (PHI) is obtained by Kuo and Trudinger for the so-called {\it implicit form} operators, see \cite[(1.16)]{KT98}.

\begin{proof}
[Proof of Theorem~\ref{Harnack}]
Let $\ell_0=1/k_0^2$ and $\ms D=\{x:|x|_\infty<R/\sqrt d\}\times[\ell_0R^2/2,\ell_0 R^2]$.
We only prove a weaker version $\sup_{\ms D}u\le C\min_{x\in B_{R/\sqrt{4d}}}u(x,0)$.
 The theorem then follows by iteration. Indeed, assume $\min_{x\in B_{R/\sqrt{4d}}}u(x,0)=u(y,0)=1$ for $y\in B_{R/\sqrt{4d}}$. Let $E_\lambda=\{\hat x \in\ms D:u(\hat x)\ge \lambda\}$. By Lemma~\ref{lem1}, $g_R(y,0;\ms D)>CR^2$. Moreover, for $s\in [\ell_0R^2/2,\ell_0R^2]$, $1=u(y,0)\ge \lambda g_R(y,0;E_\lambda\cap\{(x,t):t=s\})$, and so
 \[
 1\ge C \lambda g_R(y,0;E_\lambda)/R^2\stackrel{Theorem~\ref{thm:ap-property}(b)}{\ge} C\lambda(|E_\lambda|/|\ms D|)^c.
 \] 
Hence $|E_\lambda|/|\ms D|\le C\lambda^{-\gamma}$ for some $\gamma>0$. Therefore, for $0<p<\gamma/2$,
\begin{align*}
\norm{u}_{\ms D,p}\le \left[1+p\int_1^\infty \lambda^{p-1}|E_\lambda|/|\ms D|\dd \lambda\right]^{1/p}<C'=C'\min_{x\in B_{R/\sqrt d}}u(x,0)<\infty.
\end{align*}
This inequality, together with \cite[Theorem~A.4.1]{arXiv}, completes our proof.
\end{proof}


\section{Estimates of solutions near the boundary}\label{sec:near-bdry}
The purpose of this section is to establish estimates of $\mc L_\omega$-harmonic functions near the parabolic boundary. For $x\in\Z^d, A\subset\Z^d$, let
\[
\dist(x, A):=\min_{y\in A}|x-y|_1.
\]
\subsection{An elliptic-type Harnack inequality}
\begin{theorem}[Interior elliptic-type Harnack inequality]\label{thm-eh}
Assume $\omega\in\Omega_\kappa$. Let $R\ge 2$ and $u\ge 0$ satisfies 
\[
\left\{
\begin{array}{rl}
&\mc L_\omega u=0 \quad\mbox{ in }Q_R \\
& u=0 \quad\mbox{ in }\partial B_R\times[0,R^2).
\end{array}
\right.
\]
Then for $0<\delta\le \tfrac{1}{4}$, 
letting $Q^\delta_R:=B_{(1-\delta) R}\times[0, (1-\delta^2)R^2)$,
there exists a constant $C=C(d,\kappa, \delta)$ such that
\[
\sup_{Q^\delta_R}u\le C\inf_{Q_R^\delta}u.
\]
\end{theorem}

\begin{figure}[H]
\centering
\pgfmathsetmacro{\R}{2.5}
\pgfmathsetmacro{\r}{1}
  \begin{tikzpicture}
    \begin{axis}[
        xmin=-\R-.7, xmax=\R+1,
        ymin=-1, ymax=1.8+\R*\R,
        ticks=none, 
              xlabel=$x$,
            ylabel=$t$,
        axis lines=middle,
        unit vector ratio=2 1
      ]  
      \draw (-\R, 0) rectangle (\R,\R*\R);
      \draw[fill=blue!20, fill opacity=0.2] (-\R+\r,0) rectangle (\R-\r,\R*\R-2*\r*\r);
      \node at (0,0) [below right]{$0$};
      \node at (\R,0) [below]{\footnotesize $R$};
      \node at (0,\R*\R)[above right]{\footnotesize $R^2$};
      \node at (\R/3,\R*\R/3) {$Q_R^\delta$};
      \draw[<->, >=stealth] (\R-\r,\R*\R/2)--(\R,\R*\R/2) node[above, midway]{\footnotesize $\delta R$};
      \draw[<->, >=stealth] (-\R/2,\R*\R-2*\r*\r)--(-\R/2,\R*\R) node[midway,right]{\footnotesize $(\delta R)^2$};
     \end{axis} 
  \end{tikzpicture}
\caption{The values of $u$ are comparable inside the region $Q_R^\delta$. \label{fig:etharnack}}
\end{figure}
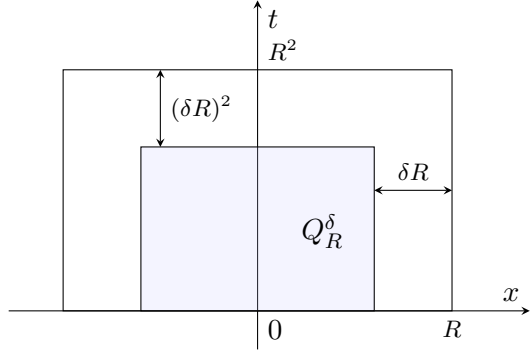

To prove Theorem~\ref{thm-eh}, we need a so-called Carlson-type estimate. For parabolic differential operators in non-divergence form, this kind of estimate was first proved by Garofalo \cite{Garo} (see also \cite[Theorem 3.3]{FSY}). 

\begin{theorem}\label{thm-Carl}
Assume $\omega\in\Omega_\kappa$, $R>2r>0$. Then for any function $u\ge 0$ that satisfies
\[
\left\{
\begin{array}{lr}
\mc L_\omega u=0 &\text{ in } (B_R\setminus\bar B_{R-2r})\times[0,3r^2)\\
u=0 &\text{ on }\partial B_R\times[r^2,3r^2)\qquad
\end{array},
\right.
\]
with the convention $\sup\emptyset=-\infty$, we have 
\begin{equation}\label{e25}
\sup_{(B_R\setminus B_{R-r})\times[r^2,2r^2)}u\le C\min_{y\in \partial B_{R-r}}u(y,0).
\end{equation}
\end{theorem}
%

\begin{proof}
Set $\ms D=(B_R\setminus\bar B_{R-2r})\times[r^2,3r^2)$. For $\hat x=(x,t)\in\ms D$, let
$d_1(\hat x)=\sup\{\rho\ge 0: B_\rho(x)\subset B_R\setminus\bar B_{R-2r}\}\ge 1$. 

First, we show that there exists $\gamma=\gamma(d,\kappa)$ such that
\begin{equation}\label{e51}
\sup_{\hat x\in\ms D}\left(
d_1(\hat x)/r
\right)^\gamma u(\hat x)\le C\min_{y\in\partial B_{R-r}}u(y,0).
\end{equation}
Indeed, for any $\hat x=(x,t)\in\ms D$, we can find 
a sequence of $n\le C\log(r/d_1(\hat x))$ balls with increasing radii $r_k:=c2^k d_1(\hat x)$:
\[
B_{r_1}(x_1)\subset B_{r_2}(x_2)\subset\cdots\subset B_{r_n}(x_n)
\subset B_R\setminus \bar B_{R-2r}
\]
such that $x_1=x$, $\dist(x_n,\partial B_{R-r})\le r/2$, and $t-r_n^2\ge r^2/2$.
By Theorem~\ref{Harnack}, 
\begin{align*}
u(x,t)
&\le Cu(x_1,t-r_1^2)\le\cdots 
\\
&\le C^n u(x_n,t-r_n^2)\le C(\frac{r}{d_1(\hat x)})^c\min_{y\in \partial B_{R-r}}u(y,0),
\end{align*}	
where in the last inequality we applied Theorem~\ref{Harnack} to a chain of parabolic balls with spatial centers at $\partial B_{R-r}$ and radius $cr$. Display \eqref{e51} is proved.

Next, with $\gamma$ as in \eqref{e51}, letting $d_0(\hat x)=\sup\{\rho\ge 0: Q_\rho(\hat x)\subset(\Z^d\setminus \bar B_{R-2r})\times[r^2,3r^2)\}$ 
, we claim that 
\begin{equation}\label{e49}
\sup_{\hat x\in\ms D}d_0(\hat x)^\gamma u(\hat x)\le \epsilon^{-\gamma}\sup_{\hat y\in\ms D}d_1(\hat y)^\gamma u(\hat y),
\end{equation}
where $\epsilon=\epsilon(d,\kappa)\in(0,1/5)$ is to be determined. 
It suffices to show that $\sup_{\ms D}d_0^\gamma u$ is achieved at $\hat x\in\ms D$ with $\epsilon d_0(\hat x)\le d_1(\hat x)$. Indeed, if $\epsilon d_0(\hat x)>d_1(\hat x)$, then $B_{2d_1(\hat x)}\setminus B_R\neq\emptyset$, and for any $\hat y=(y,s)\in Q_{2d_1(\hat x)}(\hat x)\cap\ms D$,
\[
d_0(\hat x)\le d_0(\hat y)+|x-y|+|t-s|^{1/2}\le d_0(\hat y)+4d_1(\hat x)\le d_0(\hat y)+4\epsilon d_0(\hat x)
\]
and so $d_0(\hat x)\le (1-4\epsilon)^{-1} d_0(\hat y)$. Moreover, 
by Corollary~\ref{cor:prob1}, 
\begin{align*}
d_0(\hat x)^\gamma u(\hat x)&\le 
[1-P_\omega^{\hat x}(X_\cdot \text{ exits $B_{2d_1(\hat x)}(x)\cap B_R$ from $\partial B_R$ before time }d_1^2(\hat x))]\\
&\qquad\times d_0(\hat x)^\gamma\sup_{(B_{2d_1(\hat x)}(x)\cap B_R)\times[r^2,3r^2)}u\\
&\le
(1-c_0)(1-4\epsilon)^{-\gamma}\sup_{\ms D}d_0^\gamma u
\end{align*}	
for a constant $c_0\in(0,1)$. Thus, when $\epsilon d_0(\hat x)> d_1(\hat x)$, choosing $\epsilon>0$ so that $(1-c_0)(1-4\epsilon)^{-\gamma}<1-\tfrac{c_0}{2}$, we get $d_0(\hat x)^\gamma u(\hat x)
 <(1-\tfrac{c_0}{2})\sup_{\ms D} d_0^\gamma u$.
Display \eqref{e49} is proved.
Inequality \eqref{e25} follows from \eqref{e51} and \eqref{e49}.
\end{proof}

\begin{proof}[Proof of Theorem~\ref{thm-eh}:]
Since $u=0$ on $\partial B_R\times[0,R^2)$,
\begin{align*}
\sup_{Q_R^\delta}u
&\le 
\sup_{B_R\times\{R^2-\tfrac{1}{4}(\delta R)^2\}}u\\
&\le
C\sup_{B_{(1-\delta)R}\times\{R^2-\tfrac{1}{2}(\delta R)^2\}}u\\
&\stackrel{Theorem~\ref{Harnack}}{\le} 
C(d,\kappa,\delta)\inf_{Q_R^\delta }u,
\end{align*}
where we used Theorem~\ref{thm-Carl} and Theorem~\ref{Harnack} in the second inequality.
\end{proof}

\subsection{A boundary Harnack inequality}
For positive harmonic functions with zero values on the spatial boundary, the following boundary Harnack inequality compares values near the spatial boundary and values inside, with time coordinates appropriately shifted.
\begin{theorem}[Boundary PHI]\label{thm-bh}
Let $R>0$.
Suppose $u$ is a nonnegative solution to $\mc L_\omega u=0$ in $(B_{4R}\setminus B_{2R})\times(-2R^2,3R^2)$, and $u|_{\partial B_{4R}\times\R}=0$. Then for any $\hat x=(x,t)\in (B_{4R}\setminus\bar B_{3R})\times(-R^2,R^2)$, we have
\[
C\frac{\dist(x,\partial B_{4R})}{R}\max_{y\in\partial B_{3R}}u(y, t+R^2)
\le 
u(\hat x)
\le 
C\frac{\dist(x,\partial B_{4R})}{R}\min_{y\in\partial B_{3R}}u(y,t-R^2).
\]
\end{theorem}

Theorem~\ref{thm-bh} is a lattice version of \cite[(3.9)]{Garo}. In what follows we offer a probabilistic proof.


\begin{proof}[Proof of Theorem~\ref{thm-bh}:]
Our proof uses the fact that $u(\hat X_t)$ is a martingale before exiting the region $\ms D:=(B_{4R}\setminus B_{2R})\times(-2R^2,3R^2)$.

For the lower bound, 
let $\tau_{3,4}:=\inf\{s>0: X_s\notin B_{4R}\setminus \bar B_{3R}\}$. 
By the optional stopping lemma, $u(\hat x)=E_\omega^{\hat x}[u(\hat X_{\tau_{3,4}\wedge 0.5R^2})]$, and so
\begin{align*}
u(\hat x)
&\ge P_\omega^{x,t}(\tau_{3,4}< R^2/2,X_{\tau_{3,4}}\in \partial B_{3 R}) \inf_{\partial B_{3 R}\times[t,t+0.5R^2]}
u\\
&\ge 
C\frac{\dist(x,\partial B_{4R})}{R}\max_{y\in\partial B_{3R}}u(y,t+R^2)
\end{align*}
where in the last inequality  we used Lemma~\ref{lem6} and applied Theorem~\ref{Harnack} (to a chain of parabolic balls). The lower bound is obtained.

To obtain the upper bound, note that for $\hat x\in (B_{4R}\setminus \bar B_{3R})\times(-R^2,R^2)$,
\begin{align*}
u(\hat x)
&\le 
\bigg[\max_{z\in B_{4R}\setminus\bar  B_{3R}}u(z,t+\tfrac{R^2}{2})+\max_{\partial B_{3R}\times(t,t+\tfrac{R^2}{2}]}u\bigg]
P_\omega^{x,t}(X_{\tau_{3,4}\wedge 0.5R^2}\notin \partial B_{4R})\\
&\stackrel{\eqref{e25}}{\le}
C\bigg[\max_{z\in B_{3.5R}\setminus\bar B_{3R}}u(z,t-\tfrac{R^2}{2})+\max_{\partial B_{3R}\times(t,t+\tfrac{R^2}{2}]}u\bigg]
P_\omega^{x,t}(X_{\tau_{3,4}\wedge 0.5R^2}\notin \partial B_{4R})\\
&\le
C\min_{z\in\partial B_{3R}}u(z,t-R^2)\dist(x,\partial B_{4R})/R,
\end{align*}
where in the last inequality we applied Lemma~\ref{lem7}  and used an  iteration of the Harnack inequality (Theorem~\ref{Harnack}).
\end{proof}

\section{Proof of PHI for the adjoint operator (Theorem~\ref{thm-ah})}\label{sec:proof-of-ah}

We define $\hat Y_t=(Y_t,S_t)$ to be the continuous-time Markov chain on $\Z^d\times\R$ with generator $\mc L_\omega^*$. The process $\hat Y_t$ can be interpreted as the time-reversal of $\hat X_t$. Denote by $P_{\omega^*}^{y,s}$ the quenched law of $\hat Y_\cdot$ starting from $\hat Y_0=(y,s)$ and by $E_{\omega^*}^{y,s}$ the corresponding expectation. Note that $S_t=S_0-t$.

 For $R>0, \hat x=(x,t), \hat y=(y,s)\in B_R\times\R$ with $s>t$, set 
\[
\begin{array}{rl}
& p_R^\omega(\hat x;\hat y)=P_\omega^{x,t}(X_{s-t}=y,s-t<\tau_R(\hat X)),\\
&p_R^{*\omega}(\hat y;\hat x)=P_{\omega^*}^{y,s}(Y_{s-t}=x, s-t<\tau_R(\hat Y)),
\end{array}
\]
where 
\begin{equation}\label{e22}
\tau_R(\hat X):=\inf\{t\ge 0: X_t\notin B_R\}
\end{equation}
and $\tau_R(\hat Y)$ is defined similarly.
Note that 
\[
p_R^{*\omega}(\hat y;\hat x)=\frac{\rho_\omega(\hat x)}{\rho_\omega(\hat y)}p_R^\omega(\hat x;\hat y).
\]
\begin{lemma}\label{lem8}
For any $\hat y=(y,s)\in B_R\times(0,\infty)$ and any non-negative solution $v$ of $\mc L^*_\omega v=0$ in $B_R\times(0,s]$, 
\begin{align*}
\MoveEqLeft v(\hat y)
=
\sum_{x\in\partial B_R,z\in B_R,x\sim z}\int_0^s 
\frac{\rho_\omega(x,t)}{\rho_\omega(\hat y)}\omega_t(z,x)p^{\omega}_R(z,t;\hat y) v(x,t)
\dd t\\
&+
\sum_{x\in B_R}\frac{\rho_\omega(x,0)}{\rho_\omega(\hat y)}p_R^\omega(x,0;\hat y)v(x,0).
\end{align*}

\end{lemma}
\begin{proof}
Write the two summations in the lemma as $\rom{1}$ and $\rom{2}$. 
Clearly, $\rom{2}=E_{\omega^*}^{\hat y}[v(\hat Y_s)1_{\tau_R>s}]$. 
Since $(v(\hat Y_t))_{t\ge 0}$ is a martingale, we have
 \begin{align*}
 v(y,s)
 =E_{\omega^*}^{\hat y}[v(\hat Y_{\tau_R})1_{\tau_R\le s}]
 +E_{\omega^*}^{\hat y}[v(\hat Y_s)1_{\tau_R>s}].
 \end{align*}
So it remains to show $\rom{1}=E_{\omega^*}^{y,s}[v(\hat Y_{\tau_R})1_{\tau_R\le s}]$.
We claim that for $x\in\partial B_R$,
 \begin{equation}\label{e30}
 P_{\omega*}^{\hat y}(Y_{\tau_R}=x,\tau_R\in\dd t)
=
 \sum_{z\in B_R,z\sim x}\frac{\rho_\omega(x,s-t)}{\rho_\omega(\hat y)}\omega_{s-t}(z,x)p_R^\omega(z,s-t;\hat y)\dd t.
 \end{equation}
 Indeed, for $h>0$ small enough, $x\in\partial B_R$ and almost every $t\in(0,s)$,
 \begin{align*}
&P_{\omega^*}^{y,s}(Y_{\tau_R}=x,\tau_R\in (t-h,t+h))\\
&=\sum_{z\in B_R:z\sim x}P_{\omega^*}^{\hat y}(Y_{t-h}=z,\tau_R>t-h)P_{\omega^*}^{z,s-t+h}(Y_{2h}=x)+o(h)\\
&=\sum_{z\in B_R:z\sim x}p_R^{\omega^*}(\hat y;z,s-t)
\int_{-h}^h\omega^*_{s-t+r}(z,x)\dd r+o(h).
\end{align*}
Dividing both sides by $2h$ and taking $h\to 0$, display \eqref{e30} follows by Lebesgue's differentiation theorem. Applying \eqref{e30} to 
\[
E_{\omega^*}^{y,s}[v(\hat Y_{\tau_R})1_{\tau_R\le s}]
= 
\sum_{x\in\partial B_R}\int_0^s v(x,s-t)P_{\omega*}^{\hat y}(Y_{\tau_R}=x,\tau_R\in\dd t),
\]
 we obtain $\rom{1}=E_{\omega^*}^{y,s}[v(\hat Y_{\tau_R})1_{\tau_R\le s}]$ with a change of variable.
\end{proof}

For fixed $\hat y:=(y,s)\in B_{R}\times\R$,
set $u(\hat x):=p^\omega_{2R}(\hat x,\hat y)$. Then $\mc L_\omega u=0$ in $B_{2R}\times(-\infty,s)\cup (B_{2R}\setminus B_R)\times\R$ and $u(x,t)=0$ when $x\in\partial B_{2R}$ or $t>s$. 
By Theorem~\ref{thm-bh} and Theorem~\ref{thm-eh},
for any $(x,t)\in B_{2R}\times(s-4R^2,s-\tfrac{R^2}{2})$,
\begin{align}\label{e34}
u(x,t)
\asymp u(o,s-R^2)\dist(x,\partial B_{2R})/R,
\end{align}
and, for any $(x,t)\in (B_{2R}\setminus B_{3R/2})\times(s-4R^2,s)$,
\begin{align}\label{e35}
u(x,t)
&\le Cu(o,s-R^2)\dist(x,\partial B_{2R})/R.
\end{align}

\begin{lemma}\label{lem9}
Let $v\ge 0$ satisfies $\mc L_\omega^* v=0$ in $B_{2R}\times(0,4R^2]$, then for any $\bar Y=(\bar y,\bar s)\in B_R\times(3R^2,4R^2]$ and $\lbar Y=(\lbar y,\lbar s)\in B_R\times(R^2,2R^2)$, we have
\[
\frac{v(\bar Y)}{v(\lbar Y)}
\ge 
C\dfrac{\int_{0}^{R^2}\rho_\omega(\partial B_{2R},t)\dd t
+\sum_{x\in B_{2R}}\rho_\omega(x,0)\dist(x,\partial B_{2R})}
{\int_{0}^{4R^2}\rho_\omega(\partial B_{2R},t)\dd t
+\sum_{x\in B_{2R}}\rho_\omega(x,0)\dist(x,\partial B_{2R})}.
\]
\end{lemma}

\begin{proof}
Write $\hat x:=(x,t)$ and set $\bar u(\hat x):=p^\omega_{2R}(\hat x;\bar Y)$,
$\lbar u(\hat x):=p^\omega_{2R}(\hat x;\lbar Y)$.
By Lemma~\ref{lem8} and \eqref{e34}, 
\begin{align}\label{e32}
v(\bar Y)
&\ge 
C\sum_{x\in\partial B_{2R},z\in B_{2R},x\sim z}\int_0^{\lbar s} \frac{\rho_\omega(\hat x)}{\rho_\omega(\bar Y)}\bar u(z,t)v(\hat x)\dd t
\nonumber\\
&\qquad+
C\sum_{x\in B_{2R}}\frac{\rho_\omega(x,0)}{\rho_\omega(\bar Y)}\bar u(0,\bar s-R^2)\frac{\dist(x,\partial B_{2R})}{R}v(x,0)\nn\\
&\ge
C\frac{\bar u(0,\bar s-R^2)}{R\rho_\omega(\bar Y)}\bigg[\sum_{x\in\partial B_{2R}}\int_0^{\lbar s}\rho_\omega(\hat x)v(\hat x)\dd t\nn\\
&\qquad+\sum_{x\in B_{2R}}\rho_\omega(x,0)\dist(x,\partial B_{2R})v(x,0)\bigg].
\end{align}
Similarly, by Lemma~\ref{lem8} and \eqref{e35}, we have
\begin{align}\label{e33}
v(\lbar Y)
&\le 
C\frac{\lbar u(0,\lbar s-R^2)}{R\rho_\omega(\lbar Y)}\bigg[\sum_{x\in\partial B_{2R}}\int_0^{\lbar s}\rho_\omega(\hat x)v(\hat x)\dd t\nn\\
&\qquad+\sum_{x\in B_{2R}}\rho_\omega(x,0)\dist(x,\partial B_{2R})v(x,0)\bigg].
\end{align}
Combining \eqref{e32} and \eqref{e33}, we get
\begin{equation}\label{eq:v-low-up}
\frac{v(\bar Y)}{v(\lbar  Y)}
\ge 
C\frac{\bar u(o,\bar s-R^2)/\rho_\omega(\bar Y)}{\lbar u(o, \lbar s-R^2)/\rho_\omega(\lbar Y)}.
\end{equation}
Next, taking $v\equiv 1$, by Lemma~\ref{lem8} and \eqref{e35}, 
\begin{align*}
1
&=
\sum_{x\in\partial B_{2R},z\in B_{2R},z\sim x}\int_{0}^{\bar s}
\frac{\rho_\omega(\hat x)}{\rho_\omega(\bar Y)}
\omega_t(z,x)\bar u(z,t)\dd t
+
\sum_{x\in B_{2R}}\frac{\rho_\omega(x,0)}{\rho_\omega(\bar Y)}\bar u(x,0)\\
&\le 
C\frac{\bar u(o,\bar s-R^2)}{R\rho_\omega(\bar Y)}
\big[
\sum_{x\in\partial B_{2R}}\int_{0}^{\bar s}\rho_\omega(\hat x)\dd t
+\sum_{x\in B_{2R}}\rho_\omega(x,0)\dist(x,\partial B_{2R})
\big].
\end{align*}
Similarly, by Lemma~\ref{lem8} and \eqref{e34},
\begin{align*}
1
&\ge 
C\frac{\lbar u(o,\lbar s-R^2)}{R\rho_\omega(\lbar Y)}
\big[
\sum_{x\in\partial B_{2R}}\int_{0}^{\lbar s/2}\rho_\omega(\hat x)\dd t
+\sum_{x\in B_{2R}}\rho_\omega(x,0)\dist(x,\partial B_R)
\big].
\end{align*}
These inequalities, together with \eqref{eq:v-low-up}, yield the lemma.
\end{proof}

\begin{remark}
It is clear that for static environments, the adjoint Harnack inequality (Theorem~\ref{thm-ah}) follows immediately from Lemma~\ref{lem9}. However, in time-dependent case, we need the parabolic volume-doubling property of $\rho_\omega$.
\end{remark}

\begin{proof}[Proof of Theorem~\ref{thm-ah}]
First, we will show that for all $R>0$,
\begin{align}\label{e41}
&\int_0^s\rho_\omega(\partial B_R,t)\dd t+\sum_{x\in B_R}\rho_\omega(x,0)\dist(x,\partial B_R)\\
&\asymp
\frac{1}{R}\int_0^s\rho_\omega(B_R,t)\dd t+\sum_{x\in B_R}\rho_\omega(x,s)\dist(x,\partial B_R).\nn
\end{align}
Recall $\tau_R$ at \eqref{e22} and set 
$g(x,t)=E_\omega^{x,t}[\tau_R(\hat X)]$. Note that $g(x,\cdot)=0$ for $x\notin B_R$ and $\mc L_\omega g(x,t)=-1$ if $x\in B_R$.
By \eqref{rho-invariance}, for any $s>0$,
\begin{align*}
&0=\sum_{x\in\Z^d}\int_0^s g(x,t)[\sum_y\rho_\omega(y,t)\omega_t(y,x)-\partial_t\rho_\omega(x,t)]\dd t\\
&=\sum_{x\in \partial B_R,y\in B_R}\int_0^s\rho(x,t)\omega_t(x,y)g(y,t)\dd t+\sum_{x\in B_R}g(x,0)\rho(x,0)\\
&\qquad-
\sum_{x\in B_R}\int_0^s\rho(x,t)\dd t-\sum_{x\in B_R}g(x,s)\rho(x,s).
\end{align*} 
Moreover, since  $|X_t|^2-\frac{d}{\kappa}t$ and $|X_t|^2-\kappa t$ are super- and sub- martingales, 
\[
g(x,t)\asymp 
E_\omega^{x,t}[|X_{\tau_R}|^2-|x|^2]
\asymp 
R\dist(x,\partial B_R) \quad \forall (x,t)\in B_R\times\R
\]
by the optional-stopping theorem. Display \eqref{e41} then follows.

Combining \eqref{e41} and Lemma~\ref{lem9}, we obtain
\[
\frac{v(\bar Y)}{v(\lbar Y)}
\ge 
C\frac{\int_0^{R^2}\rho(B_{2R},t)\dd t+R\sum_{x\in B_{2R}}\rho(x,R^2)\dist(x,\partial B_{2R})}
{\int_0^{4R^2}\rho(B_{2R},t)\dd t+R\sum_{x\in B_{2R}}\rho(x,4R^2)\dist(x,\partial B_{2R})}.
\]
Finally, Theorem~\ref{thm-ah} follows by Theorem~\ref{thm:vd} and the above inequality.
\end{proof}

\section{Proof of Theorem~\ref{thm:hke}, Corollary~\ref{cor:q-estimates} and Theorem~\ref{thm:llt}}
\label{sec:proof-llt-hke}

\subsection{Proof of Theorem~\ref{thm:hke}}
\begin{proof}
First, using Theorem~\ref{thm-ah} and standard arguments, 
we will prove \eqref{eq:quenched-hke}.
Recall that $v(\hat x):=q^\omega(\hat 0,\hat x)$ satisfies $\mc L_\omega^* v=0$ in $\Z^d\times(0,\infty)$. By Theorem~\ref{thm-ah}, for $\hat x=(x,t)\in \Z^d\times(0,\infty)$, we have
$v(\hat x)\le C\min_{y\in B_{\sqrt t}(x)}v(y,3t)$
 and so
\begin{align*}
v(\hat x)
&\le 
\frac{C}{\rho(B_{\sqrt t}(x), 3t)}\sum_{y\in B_{\sqrt t}(x)}\rho(y,3t)v(y,3t)\\
&=
\frac{C}{\rho(B_{\sqrt t}(x), 3t)}P_\omega^{0,0}(X_{3t}\in B_{\sqrt t}(x))
\le 
C\exp[-c\mathfrak{h}(|x|,t)],
\end{align*}
where Corollary~\ref{cor:prob-upperb} is used in the last inequality.
Moreover, for any $s\in[0,t]$,  $|y|\le|x|+c\sqrt t$, by Theorem~\ref{thm:vd} and iteration,
\[
\rho(B_{\sqrt t}(x),3t)\ge C \rho(B_{\sqrt {t\vee1}}(x),s)\ge
C\left(\tfrac{|x|}{\sqrt{t\vee1}}+1\right)^{-c}
\rho(B_{\sqrt{t\vee1}}(y),s).
\]
Since $\tfrac{|x|}{\sqrt{t\vee1}}+1\le C_\epsilon e^{\epsilon\mathfrak{h}(|x|,t)}$ for any $\epsilon>0$, the upper bound in \eqref{eq:quenched-hke} follows.

To obtain the lower bound in \eqref{eq:quenched-hke}, by similar argument as above and Theorem~\ref{thm-ah}, 
$v(\hat x)\ge C\max_{y\in B_{\sqrt t/2}(x)}v(y,t/4)$ for $\hat x\in\Z^d\times(0,\infty)$, and so
\begin{equation}\label{e52}
v(\hat x)
\ge 
\frac{C}{\rho(B_{\sqrt t/2}(x), t/4)} P_\omega^{0,0}(X_{t/4}\in B_{\sqrt t/2}(x)).
\end{equation}
We claim that for any $(y,s)\in \Z^d\times(0,\infty)$,
\begin{equation}\label{hitting-upperb}
P_\omega^{y,0}(X_s\in B_{\sqrt s})\ge Ce^{-c|y|^2/s}.
\end{equation}
Indeed, the case $|y|/\sqrt s\le 3$ follows from Lemma~\ref{lem1}.  When $|y|/\sqrt s>3$, let
\[
n=\floor{2|y|^2/s}. 
\]
Set $u(x,t):=p^\omega(x,t;B_{\sqrt s},s)$. Then $u$ is a $\mc L_\omega$-harmonic function on $\Z^d\times(-\infty,s)$. Taking a sequence of points $(y_i)_{i=1}^n$ such that $y_0=y, y_n=0$ and $|y_i-y_{i+1}|\le |y|/n$, for $i=0,\ldots n-1$,
\begin{align*}
&\min_{x\in B_{|y|/\sqrt{n}}(y_i)}u(x,\tfrac{i|y|^2}{n^2})\\
&\ge 
\min_{z\in B_{|y|/\sqrt{n}}(y_i)}p^\omega(z,\tfrac{i|y|^2}{n^2};B_{|y|/\sqrt{n}}(y_{i+1}),\tfrac{(i+1)|y|^2}{n^2})\min_{x\in B_{|y|/\sqrt{n}}(y_{i+1})}u(x,\tfrac{(i+1)|y|^2}{n^2})\\
&\stackrel{Lemma~\ref{lem1}}{\ge}
C\min_{x\in B_{|y|/\sqrt{n}}(y_{i+1})}u(x,\tfrac{(i+1)|y|^2}{n^2}). 
\end{align*}	
Iteration then yields $u(y,0)\ge C^{n-1}\min_{x\in B_{|y|/\sqrt{n}}}u(x,\tfrac{|y|^2}{n})\stackrel{Lemma~\ref{lem1}}{\ge} C^n$.
Inequality \eqref{hitting-upperb} is proved. Then, by \eqref{e52},
\[
v(x,t)\ge 
\frac{C}{\rho(B_{\sqrt t/2}(x), t/4)}e^{-c|x|^2/t}.
\]
Moreover, by Theorem~\ref{thm:vd}, we have for any $s\in[0,t], |y|\le |x|$,
\[
\rho(B_{\sqrt t/2}(x), t/4)
\le C\rho(B_{\sqrt t/2}(x),s)
\le C(\tfrac{|x|}{\sqrt t}+1)^c\rho(B_{\sqrt t}(y),s).
\]
The  lower bound in \eqref{eq:quenched-hke} is proved.

Next,  we will prove the moment bounds \eqref{eq:mm-hke1} and \eqref{eq:mm-hke2}, which, by \eqref{eq:quenched-hke} and \eqref{eq:rho-asymp}, are equivalent to showing that, for $r:=\sqrt t>0$,
\[
\Norm{\frac{\rho(\hat 0)}{\rho(Q_{r})}}_{L^{(d+1)/d}(\mb P)}
\le C r^2(r\vee1)^{-d}
\quad\text{ and }\quad
\Norm{\frac{\rho(\hat 0)}{\rho(Q_{r})}}_{L^{-p}(\mb P)}
\ge Cr^2(r\vee1)^{-d},
\]
where $Q_r$ is as defined in \eqref{def:parab-ball}.
Indeed, using the translation-invariance of $\mb P$ and the volume-doubling property of $\rho$, for $q:=(d+1)/d$,
\begin{align*}
\norm{\rho(\hat 0)/\rho(Q_r)}_{L^{q}(\mb P)}^q
\le 
C\frac{1}{|Q_r|}\int_{\hat x\in Q_r}E_{\mb P}\left[
\frac{\rho(\hat x)^q}{\rho(Q_r)^q}
\right]
\le 
C/|Q_r|^q,
\end{align*}	
where we used the Reverse H\"older inequality (Corollary~\ref{cor:reverse-holder}(i)) in the last inequality. 
Recalling \eqref{eq:def-integral}, inequality \eqref{eq:mm-hke1} then  follows from the fact that $|Q_r|=r^2\sum_{x\in B_r}1\asymp r^2(r\vee 1)^{-d}$.

To obtain \eqref{eq:mm-hke2}, note that by translation invariance and $\mb P$ and the volume-doubling property of $\rho$, taking $\epsilon\in(0,1/(p-1))$,
\begin{align*}
\norm{\rho(Q_r)/\rho(\hat 0)}_{L^\epsilon(\mb P)}^\epsilon
\le 
\frac{C}{|Q_r|}E_{\mb P}\left[\int_{\hat x\in Q_r}\frac{\rho(Q_{r})^\epsilon}{\rho(\hat x)^\epsilon}\right]
\le
C|Q_r|^\epsilon,
\end{align*}	
where we used the $A_p$ inequality (Theorem~\ref{thm:ap-property}\eqref{item:neg-moment}) of $\rho$ in the last inequality.
Therefore $\norm{\rho(\hat 0)/\rho(Q_r)}_{L^{-\epsilon}(\mb P)}\ge Cr^2(r\vee 1)^{-d}$ and \eqref{eq:mm-hke2} is proved.

Display \eqref{eq:mm-green} follows from \eqref{eq:mm-hke1}, \eqref{eq:mm-hke2}, and Minkowski's integral inequality.
\end{proof}

\subsection{Proof of Theorem~\ref{thm:llt}}\label{subsec:pf-thm-llt}
As a standard consequence of the PHI for $\mc L^*_\omega$, we first state the following H\"older estimate. (See a proof in \cite[Section~A.2]{arXiv}.)
\begin{corollary}\label{cor:hoelder} 
There exists $\gamma=\gamma(d,\kappa)\in(0,1]$  such that for
$\mb P$-almost all $\omega$, 
any  non-negative solution $u$ of $\mc L_\omega^*$ in $B_R(x_0)\times(t_0-R^2,t_0]$, $R>0$,  satisfies 
\[
|u(\hat x)-u(\hat y)|\le C
\left(
\frac{r}{R}
\right)^\gamma \sup_{ B_R(x_0)\times(t_0-R^2,t_0]}u
\]
for all $\hat x,\hat y\in B_r(x_0)\times(t_0-r^2,t_0]$ and $r\in(0,R)$.
\end{corollary}

Recall $q^\omega(\hat y,\hat x)$ in \eqref{eq:def-hk}. For any $\hat x=(x,t)\in\R^d\times\R$, set 
\[
v(\hat x):=q^\omega(\hat 0;\floor{x},t),
\]
where $\floor{x}$ is as in Theorem~\ref{thm:llt}. Note that $\mc L_\omega^* v=0$ in $\Z^d\times(0,\infty)$. By Corollary~\ref{cor:hoelder} and Theorem~\ref{thm:hke}, for any $\hat y=(y,s)\in B_{\sqrt t}(x)\times(\tfrac t2, t)$,
\begin{align}
\label{eq:holder-v}
|v(\hat x)-v(\hat y)|
&\le C\left(\frac{|x-y|+\sqrt{t-s}}{\sqrt t}\right)^\gamma \sup_{B_{\sqrt t}(x)\times(\tfrac t2, t]}v\nn\\
&\le
C\left(\frac{|x-y|+\sqrt{t-s}}{\sqrt t}\right)^\gamma t^{-d/2}
\end{align}
when $t>t_0(\omega)$ is big enough. Here in the last inequality we used  Corollary~\ref{cor:q-estimates}(i) which is an immediate consequence of Theorem~\ref{thm:hke} and \eqref{eq:rho-ergodic}. 

Recall $\mc O_r$ in \eqref{eq:def-continuousball}.  For $\hat x=(x,t)\in\R^d\times\R$, write 
\[
\hat x^{n}:=(\floor{nx},n^2t).
\]
To prove Theorem~\ref{thm:llt}, it suffices to show that for any $K>T$,
\begin{align}
\label{eq:uniform-conv}
\lim_{n\to\infty}\sup_{\hat x\in\mc O_K\times[T,K]}|n^dv(\hat x^n)-p_t^\Sigma(0,x)|=0.
\end{align}
Indeed, for any $\epsilon>0$, there exists $K=K(T,\epsilon,d,\kappa)>0$ such that, writing $\ms D:=(\R^d\times[T,\infty))\setminus(\mc O_K\times[T,K])$,  we have
\begin{align*}
\varlimsup_{n\to\infty}\sup_{\ms D}
n^dv(\hat x^n)+p_t^\Sigma(0,x)
&\stackrel{\eqref{eq:quenched-hke},\eqref{eq:rho-ergodic}}{\le} 
C\sup_{\ms D}
t^{-d/2}e^{-c|x|^2/t}
\le \epsilon.
\end{align*}
 Hence 
Theorem~\ref{thm:llt} follows provided that \eqref{eq:uniform-conv} is proved.

\begin{proof}[Proof of Theorem~\ref{thm:llt}]
As we discussed in the above, it suffices to prove \eqref{eq:uniform-conv}. It suffices to consider the case $T<K<2T$.
For any  $\epsilon>0$, 
\begin{align}\label{eq:split}
&\Abs{n^d v(\hat x^n)-p_t^\Sigma(0,x)}
\le A(\hat x,\epsilon)+B_n(\hat x,\epsilon)+C_n(\hat x,\epsilon),
\end{align}	
where $A(\hat x,\epsilon)=\Abs{\frac{\int_t^{t+\epsilon^2}p_s^\Sigma(0,\mc O_\epsilon(x))\dd s}{\epsilon^2|\mc O_\epsilon|}-p_t^\Sigma(0,\mc O_\epsilon(x))}$, 
\[B_n(\hat x,\epsilon)=\Abs{\int_t^{t+\epsilon^2}\frac{P_\omega^{\hat 0}(X_{n^2s}\in n\mc O_\epsilon(x))-p_s^\Sigma(0,\mc O_\epsilon(x))}{\epsilon^2|\mc O_\epsilon|}\dd s},\]
\[C_n(\hat x,\epsilon)=\Abs{n^d v(\hat x^n)-\frac{\int_t^{t+\epsilon^2}P_\omega^{\hat 0}(X_{n^2s}\in n\mc O_\epsilon(x))\dd s}{\epsilon^2|\mc O_\epsilon|}}.
\]


First, we will show that 
\begin{align}
\label{eq:uniform-ergo}
\varlimsup_{n\to\infty}\sup_{\hat x\in\mc O_K\times[T,K]}C_n(\hat x,\epsilon)=O(\epsilon^\gamma).
\end{align}
To this end, note that there exists $N=N(T,\omega,d,\kappa)$ such that for $n\ge N$,
\begin{align*}
C_n(\hat x,\epsilon)&\le 
n^dv(\hat x^n)\Abs{1-\frac{\int_t^{t+\epsilon^2}\rho(n\mc O_\epsilon(x),n^2s)\dd s}{\epsilon^2|n\mc O_\epsilon|}}
\\&
+\sum_{y\in n\mc O_\epsilon(x)}\int_t^{t+\epsilon^2}|v(y,n^2s)-v(\hat x^n)|\rho(y,n^2s)\dd s/(\epsilon^2|\mc O_\epsilon|)\\
&\le 
CT^{-d/2}\Abs{1-\frac{\int_t^{t+\epsilon^2}\rho(n\mc O_\epsilon(x),n^2s)\dd s}{\epsilon^2|n\mc O_\epsilon|}}\\&\qquad
+CT^{-(\gamma+d)/2}\epsilon^{\gamma}\int_t^{t+\epsilon^2}\rho(n\mc O_\epsilon(x),n^2s)\dd s/(\epsilon^2|n\mc O_\epsilon|),
\end{align*}	
where in the second inequality we used Corollary~\ref{cor:q-estimates}(i) and \eqref{eq:holder-v}. 
Further, by an ergodic theorem of Krengel and Pyke \cite[Theorem 1]{KP87} and  \eqref{rho-moment},
\begin{equation}\label{eq:uniform-rho}
\lim_{n\to 0}\sup_{\hat x\in \mc O_K\times[T,K]}
\Abs{1-\frac{\int_t^{t+\epsilon^2}\rho(n\mc O_\epsilon(x),n^2s)\dd s}{\epsilon^2|n\mc O_\epsilon|}}=0.
\end{equation}
Display \eqref{eq:uniform-ergo} follows.

Next, for $\hat x=(x,t)$, by writing $B_n(\hat x,\epsilon)$ as
\[
\Abs{\frac{\int_t^{t+\epsilon^2}P_\omega^{\hat 0}(X_{n^2s}\in n\mc O_\epsilon(x))\dd s}{\epsilon^2|\mc O_\epsilon|}-\frac{\int_t^{t+\epsilon^2}p_s^\Sigma(0,\mc O_\epsilon(x))\dd s}{\epsilon^2|\mc O_\epsilon|}}=:|B_n^1(\hat x,\epsilon)-B^2(\hat x,\epsilon)|,
\]
we will show that
\begin{equation}\label{eq:equicontin}
\varlimsup_{n\to\infty}\sup_{\hat x\in \mc O_K\times[T,K]}B_n(\hat x,\epsilon)=O(\epsilon^\gamma).
\end{equation}
We claim that $B_n(\hat x,\epsilon)$ is approximately equicontinuous (with order $\epsilon^\gamma$). That is, there exist $N,\delta$ depending on $(\epsilon,\omega, d,\kappa,T,K)$ such that, whenever $n\ge N$ and $\hat x_1=(x_1,t_1), \hat x_2=(x_2,t_2)\in\mc O_K\times[T,K]$ satisfy $|\hat x_1-\hat x_2|_1:=|x_1-x_2|+|t_1-t_2|<\delta$, we have
\[
|B_n(\hat x_1,\epsilon)-B_n(\hat x_2,\epsilon)|<C\epsilon^\gamma.
\]
It suffices to show that $B_n^1(\hat x,\epsilon)$ is approximately equicontinuous. Indeed, 
by \eqref{eq:uniform-ergo} and \eqref{eq:holder-v}, 
when $n\ge N$ is large and $\hat x_1,\hat x_2\in \mc O_K\times[T,K]$,
\begin{align*}
|B^1_n(\hat x_1,\epsilon)-B^1_n(\hat x_2,\epsilon)|
&\le
C_n(\hat x_1,\epsilon)+C_n(\hat x_2,\epsilon)+ n^d|v(\hat x_1^n)-v(\hat x_2^n)|\\
&\le C\epsilon^\gamma+C(|x_1-x_2|+\sqrt{|t_1-t_2|})^\gamma.
\end{align*}	
The approximate equicontinuity of $B_n^1(\hat x,\epsilon)$ follows. To prove \eqref{eq:equicontin}, we choose a finite sequence $\{\hat x_i\}_{i=1}^M$ such that $\min_{1\le i\le M}|\hat x-\hat x_i|_1<\delta$ for all $\hat x\in \mc O_K\times[T,K]$. Since $\lim_{n\to\infty}\max_{1\le i\le M}B_n(\hat x_i)=0$ by the QCLT (Theorem~\ref{thm:recall}), display \eqref{eq:equicontin} follows by the approximate equicontinuity.

Clearly, $\lim_{\epsilon\to 0}\sup_{\hat x\in\mc O_K\times[T,K]}A(\hat x,\epsilon)=0$. This, together with \eqref{eq:uniform-ergo} and \eqref{eq:equicontin}, yields the uniform convergence of \eqref{eq:split} by sending first $n\to\infty$ and then $\epsilon\to 0$. Our proof of Theorem~\ref{thm:llt} is complete.
\end{proof}

\begin{proof}[Proof of Corollary~\ref{cor:q-estimates}:]

\eqref{cor:q-hke}
follows from Theorem~\ref{thm:hke} and \eqref{eq:rho-ergodic}.
\eqref{cor:green1} and \eqref{cor:green2} are consequences of Theorems~\ref{thm:llt} and \ref{thm:hke}. Their proofs, which are similar to 
\cite[Theorem 1.14]{ACDS} and 
 \cite[Theorem 1.4]{ADS18}, can be found in \cite[A.6]{arXiv}.
\end{proof}

\section{Auxiliary probability estimates}\label{sec:auxiliary-prob}
This section contains probability estimates that are useful in the rest of the paper. It does not rely on results in the previous sections, and can be read independently. Recall the definition of the function $\mathfrak{h}(r,t)$ in \eqref{eq:def-function-h}. 
\begin{theorem}\label{thm:fluctuation}
Assume $\omega\in\Omega_\kappa$. Then for $t>0$, $r>0$,
\[
P_\omega^{0,0}(\sup_{0\le s\le t}|X_s|\ge r)\le C\exp\left(-c\mathfrak{h}(r,t)\right).
\]
\end{theorem}
\begin{proof}
Let $x(i), i=1,\ldots,d,$ denotes the $i$-th coordinate of $x\in\R^d$.
It suffices to show that for $i=1,\ldots,d$,
\[
P_\omega^{0,0}(\sup_{0\le s\le t}|X_s(i)|>r)
\le C\exp\left(-c\mathfrak{h}(r,t)\right).
\]
We will prove the statement for $i=1$.
Let $\tilde N_t:=\#\{0\le s\le t: X_s(1)\neq X_{s^-}(1)\}$ be the number of jumps in the $e_1$ direction before time $t$. Let $(S_n)$ be the discrete time simple random walk on $\Z$, then  $X_t(1)\stackrel{d}{=}S_{\tilde N_t}$. 
Note that  $\tilde N_t$ is stochastically dominated by a Poisson process $N_t$ with rate $c_0:=2d/\kappa$, and so $P_\omega^{\hat 0}(\sup_{0\le s\le t}|X_s(1)|>r)\le P(\max_{0\le m\le N_t}|S_m|>r)$. 
Hence, 
\begin{align*}
P_\omega^{\hat 0}(\sup_{0\le s\le t}|X_s(1)|>r)
&\le P(N(t)\ge 2c_0(t\vee r))+P(\max_{0\le m\le 2c_0(t\vee r)}|S_m|>r)\\
&\le e^{-c(t\vee r)}+Ce^{-cr^2/(t\vee r)}\le Ce^{-cr^2/(t\vee r)}.
\end{align*}	
where we used Hoeffding's inequality in the second inequality. On the other hand, since the random walk is in a discrete set $\Z$, we have, for any $\theta>0$,
\begin{align*}
P_\omega^{\hat 0}(\sup_{0\le s\le t}|X_s(1)|>r)
&\le 
P(N(t)>r)\\
&\le E[\exp(\theta N(t)-\theta r)]
=\exp[
c_0t(e^\theta-1)-\theta r].
\end{align*}	
When $r\ge 9c_0^2 t$, taking $\theta=\log(\tfrac{r}{c_0t})$, we get an upper bound
$\exp[-\tfrac r2\log(\tfrac{r}{t})]$. Hence, letting $f(r,t)=\tfrac{r^2}{t\vee r}\mathbbm{1}_{r<9c_0^2t}+r\log(\tfrac{r}{t})\mathbbm{1}_{r\ge 9c_0^2t}$, we obtain
\[
P_\omega^{\hat 0}(\sup_{0\le s\le t}|X_s(1)|>r)
\le 
C\exp(-cf(r,t)).
\]
Since
$f(r,t)\asymp \tfrac{r^2}{t\vee r}+r\log(\tfrac{r}{t})\mathbbm{1}_{r\ge 9c_0^2t}\asymp \mathfrak{h}(r,t)$, our proof is complete.
\end{proof}

\begin{corollary}\label{cor:prob-upperb}
Assume $\omega\in\Omega_\kappa$ and $\theta_2>\theta_1>0$.
There exist $C,c$ depending on $(d,\kappa,\theta_1,\theta_2)$ such that for  $\theta\in(\theta_1,\theta_2), (x,t)\in\Z^d\times(0,\infty)$,
\begin{equation*}
P_\omega^{0,0}(X_{t}\in B_{\sqrt{\theta t}}(x))
\le 
C\exp[-c\mathfrak{h}(|x|,t)].
\end{equation*}
\end{corollary}
\begin{proof}
Since $\mathfrak{h}(0,t)=0$, we only need to consider the case $x\neq 0$. 

If $\theta t\le 1$, then $P_\omega^{\hat 0}(X_{t}\in B_{\sqrt{\theta t}}(x))=P_\omega^{\hat 0}(X_t=x)\le P_\omega^{\hat 0}(\sup_{0\le s\le t}|X_s|\ge |x|)\le C\exp[-c\mathfrak{h}(|x|,t)]$ by Theorem~\ref{thm:fluctuation}. 

If $\theta t>1$ and $1\le |x|\le 2\sqrt{\theta t}$, then $|x|\le |x|^2\le 4\theta t$ and so $\mathfrak{h}(|x|,t)\asymp \mathfrak{h}(|x|,4\theta t)\asymp\tfrac{|x|^2}{t}$. In particular, $\mathfrak{h}(|x|,t)\le C\tfrac{|x|^2}{t}\le c$. Hence, trivially, $P_\omega^{\hat 0}(X_{t}\in B_{\sqrt{\theta t}}(x))\le 1\le C\exp(-c\mathfrak{h}(|x|,t))$.

It reminds to consider $|x|>2\sqrt{\theta t}>2$. In this case, by Theorem~\ref{thm:fluctuation},
$P_\omega^{\hat 0}(X_{t}\in B_{\sqrt{\theta t}}(x))\le P_\omega^{\hat 0}(\sup_{0\le s\le t}|X_s|\ge |x|/2)
\le C\exp[-c\mathfrak{h}(|x|,t)]$.
\end{proof}

\begin{lemma}\label{lem1}
Let $0<\theta_1<\theta_2$, $R>0$ and $\omega\in\Omega_\kappa$.  Recall the definition of the stopping time $\Delta$ in \eqref{def:st-a-s}. 
There exists a constant $\alpha=\alpha(\kappa,d,\theta_1,\theta_2)\ge 1$ such that for any $s\in(\theta_1 R^2,\theta_2R^2)$ and $\sigma>0$
\[
\min_{x\in B_R}P_\omega^{x,0}(X_{\Delta(B_{2R},s)}\in B_{\sigma R})\ge (\frac{\sigma\wedge 1}{2})^\alpha.
\]
\end{lemma}

\begin{proof}
It suffices  to consider the case $\sigma\in(0, 1)$ and $R\ge K_1$, where $K_1=K_1(\theta_1,\theta_2,\kappa,d)$ is a large constant to be determined. Indeed, if $R<K_1$, then by uniform ellipticity, for any $x\in B_R$, 
\[
P_\omega^{x,0}(X_{\Delta(B_{2R},s)}\in B_{\sigma R})
\ge 
P_\omega^{x,0}(X_s=0, \Delta(B_{2R},s)=s)
> 
C(\kappa,d,\theta_1,\theta_2). 
\]
Further, for $R\ge K_1$, if suffices to consider the case $\sigma R\ge \sqrt{K_1}$. Indeed, assume the lemma is proved for $R\ge K_1$ and $\sigma R\ge \sqrt{K_1}$. Then, when $\sigma R<\sqrt{K_1}$ and $x\in B_R$, by uniform ellipticity, 
\begin{align*}
&P_\omega^{x,0}(X_{\Delta(B_{2R},s)}\in B_{\sigma R})\\
&\ge 
P_\omega^{x,0}(X_{\Delta(B_{2R},s-K_1)}\in B_{\sqrt K_1})\min_{y\in B_{\sqrt K_1}}P_\omega^{y,s-K_1}(X_{K_1}=0, \Delta(B_{2R},s)=K_1)\\
&\ge 
(\frac{\sqrt{K_1}}{2R})^\alpha C(K_1,\kappa,d)\ge C(\frac{\sigma}{2})^\alpha.
\end{align*}	
Hence in what follows we only consider the case
$R\ge K_1$ and $\sigma R\ge \sqrt{K_1}$. 

For $(x,t)\in\R^d\times\R$, set
\[
\psi_0(t)=1-\dfrac{1-(\sigma/2)^2}{s}t, \quad
\tilde\psi_1(x,t)=\psi_0-\frac{|x|^2}{4R^2},\quad
\psi_1=\tilde\psi_1\vee 0,
\]
and, for some large constant $q\ge 2$ to be chosen,
\[
\psi(x,t):=\psi_1^2\psi_0^{-q}, \quad w(x,t)=(\sigma/2)^{2q-4}\psi(x,t).
\]


\begin{figure}[H]
\centering
\pgfmathsetmacro{\R}{2}
\pgfmathsetmacro{\r}{1.1}
   \begin{tikzpicture}
    \draw[gray!70,->, >=stealth] (-.8-\R,0)--(\R+.8,0) node[black, right] {$x$};
    \draw[gray!70,->,>=stealth] (0,0) node[black, above right] {\small $0$}--(0,4) node[black,above] {$t$};
     \draw[line width=0.6mm] (-\R,0) coordinate (a) to (-\r,3) coordinate (b) to (\r,3) coordinate (c) to (\R,0) coordinate (d);
        \draw[gray,decorate,decoration={brace,amplitude=10pt}, yshift=1 pt]
(b) -- (c) node[black,above, midway, yshift=8pt] {$B_{\sigma R}$};
    \draw[gray,decorate,decoration={brace,amplitude=10pt,mirror},yshift=-2pt]
(-\R,0) -- (\R,0) node[black,below,midway, yshift=-6pt] {$B_{2R}$};
    \node[below right] at (0,3) {\small $s$};
   \end{tikzpicture}
\caption{The set $U$. The solid line is $\partial^\p U$.}   
\end{figure}

Let $U:=\{\hat x\in B_{2R}\times[0,s): \psi_1(\hat x)>0\}$.
We will show that for $\hat x\in U$,
\[
w(\hat x)\le v(\hat x):=P_\omega^{\hat x}(X_\tau\in B_{\sigma R}).
\]

Recall the parabolic boundary in page~\pageref{page:bdry}. We first show that $w$ satisfies
\begin{equation}\label{e7}
\begin{cases}
w|_{\partial^\p U}\le \mathbbm{1}_{x\in B_{\sigma R}, s=t}\\
\min_{x\in B_R} w(x,0)\ge \frac{1}{2}(\sigma/2)^{2q-4},\\
\mc L_\omega w\ge 0
\quad\mbox{ in $U$, for }q\mbox{ large}.
\end{cases}
\end{equation}
The first two properties in \eqref{e7} are obvious. For the third property, note that 
\begin{align*}
\partial_t\psi=R^{-2}\psi_0^{-q}[\dfrac{1-(\sigma/2)^2}{s/R^2}(q\tfrac{\psi_1}{\psi_0}-2)\psi_1] 
\quad\text{ in }U.
\end{align*}
For any unit vector $e\in\Z^d$, let 
\begin{equation}\label{eq:2nd-difference}
\nabla_e^2u(x,t):=u(x+e,t)+u(x-e,t)-2u(x,t).
\end{equation} 
When $\hat x\in U_1:=\{(z,s)\in U:(y,s)\in U \text{ for all }y\sim z\}$, then $\nabla_e^2 [\psi_1^2(\hat x)]=\nabla_e^2[\tilde\psi_1^2(\hat x)]$. When $\hat x=(x,t)\in U\setminus U_1$, then for some $|e|=1$, either $(x+e,t)$ or $(x-e,t)$ is not in $U$. Say, $(x+e,t)\notin U$, then $x\cdot e\ge 1$ and $\exists \delta\in(0,1)$ such that $\tilde\psi_1(x+\delta e,t)=0$. In both cases, there exists $\delta\in(0,1]$ such that
\begin{align*}
&\nabla^2_e[\psi_1(x,t)^2]\\
&=\tilde\psi_1^2(x+\delta e,t)+\tilde\psi_1^2(x-e,t)-2\psi_1^2(\hat x)\\
&=-\frac{[1+\delta^2+2x\cdot e(\delta -1)]\psi_1}{2R^2}+\frac{1+\delta^2+4(\delta^2+1)(x\cdot e)^2+4(\delta^3-1)x\cdot e}{16 R^4}\\
&\ge -\frac{\psi_1}{R^2}+\frac{(x\cdot e)^2}{4R^4}-\frac{\psi_0^{1/2}}{2R^3},
\end{align*}	
where in the last inequality we used the fact $1\le x\cdot e\le |x|\le 2R\psi_0^{1/2}$. 
Thus, letting $\xi:=\psi_1/\psi_0\in[0,1]$, we have for $\hat x=(x,t)\in\tilde U$,
\begin{align*}
R^2\psi_0^{q-1}\mc L_\omega\psi(\hat x)
&=R^2\left(\sum_{i=1}^d\omega_t(x,x+e_i)\nabla^2_{e_i}[\psi_1^2]/\psi_0+\psi_0^{q-1}\partial_t\psi\right)\\
&\ge \frac{c|x|^2}{R^2\psi_0}-C\xi
-\frac{C}{R\psi_0^{1/2}}
+\dfrac{1-(\sigma/2)^2}{s/R^2}(q\xi-2)\xi\\
&\ge Cq\xi^2-c_1\xi+c_2-c_3/K_1^{1/2},
\end{align*}	
where in the last inequality we used $|x|^2/(4R^2\psi_0)=1-\xi$ and $\psi_0^{1/2}\ge \sigma/2\ge K_1^{1/2}/(2R)$.
Taking $q$ and $K_1$ large enough, we have $\mc L_\omega\psi\ge 0$ in $U$. 
The third property in \eqref{e7} is proved.

Finally, we set $v(\hat x)=P_\omega^{\hat x}(X_{\Delta(B_{2R},s)}\in B_{\sigma R})$. 
By \eqref{e7},  $v(\hat X_t)-w(\hat X_t)$ is a super-martingale for 
$t\le T_U:=\inf\{s\ge 0: \hat X_s\notin U\}$ and $(v-w)|_{\partial^\p U}\ge 0$. Hence, the optional-stopping theorem yields
\[
v(\hat x)-w(\hat x)\ge E^{\hat x}_a[v(\hat X_{T_U})-w(\hat X_{T_U})]\ge 0 \quad \text{ for }\hat x\in U.
\]
In particular, 
$
\min_{x\in B_{R}}v(x,0)\ge 
\min_{x\in B_{R}}w(x,0)
\ge 
(\sigma/2)^{2q-4}/2.
$
\end{proof}

\begin{corollary}\label{cor:prob1}
Assume that $\omega\in\Omega_\kappa$, $R/2>r>1/2$, $\theta>0$.
There exists $c=c(d,\kappa,\theta)\in(0,1)$ such that for any $y\in\partial B_R$ with $B_r(y)\cap B_R\neq\emptyset$,
\[
\min_{x\in B_r(y)\cap B_R}P_\omega^{x,0}\left(X_\cdot \mbox{ exits $B_{2r}(y)\cap B_R$ from $\partial B_R$ before time } \theta r^2\right)>c.
\]
\end{corollary}

\begin{proof}

By uniform ellipticity, it suffices to prove the lemma for all $r\ge 10$.

Let $z=\tfrac{ry}{5|y|}+y\in\R^d$. Note that $B_{r}(y)\subset B_{5r/4}(z)\subset B_{3r/{2}}(z)\subset B_{2r}(y)$ and $B_{r/5}(z)\subset\Z^d\setminus B_R$. Recall $\Delta$ in \eqref{def:st-a-s}. 
Then, by Lemma~\ref{lem1},
\begin{align*}
 &\min_{x\in B_{r}(y)\cap B_R}P_\omega^{x,0}(X_{\Delta(B_{2r}(y)\cap B_R,\theta r^2)}\in \partial B_R)\\
 &\ge 
 \min_{x\in B_{5r/4(z)}}P_\omega^{x,0}(X_{\Delta(B_{3r/2}(z),\theta r^2)}\in B_{r/5}(z))
\ge c(\theta,d,\kappa).
 \end{align*}	
The corollary is proved.  
\end{proof}

\begin{lemma}\label{lem6}
Assume $\omega\in\Omega_\kappa$, $\beta\in(0,1)$. Let $\tau_{\beta,1}=\tau_{\beta,1}(R)=\inf\{t\ge 0: X_t\notin B_R\setminus\bar B_{\beta R}\}$. Then if $y\in B_R\setminus\bar B_{\beta R}\neq\emptyset$ and $\theta>0$,  we have
\[
P^{y,0}_\omega(X_{\tau_{\beta,1}}\in\partial B_{\beta R},\tau_{\beta,1}\le \theta R^2)
\ge 
C\dfrac{\dist(y,\partial B_R)}{R},
\]
where $C=C(\kappa,d,\beta,\theta)$.
\end{lemma}

\begin{proof}
It suffices to prove the lemma for $R>\alpha^2$, where $\alpha=\alpha(\kappa,d,\beta,\theta)$ is a large constant to be determined.
We only need to consider $y$ with 
$\dist(y,\partial (B_R\setminus\bar B_{\beta R}))\ge2$ in which case $R-|y|\asymp \dist(y,\partial B_R)$.

For $\hat x=(x,t)$, let $g(\hat x)=\exp(-\tfrac{\alpha}{R^2}|x|^2-\frac{\alpha t}{\theta R^2})$. 
Using the inequalities $e^a+e^{-a}\ge 2+a^2$ and $e^a\ge 1+a$,  we get for $x\in B_R\setminus \bar B_{\beta R}, t\in\R$,
\begin{align*}
\mc L_\omega g(\hat x)
&=g(\hat x)\left(\sum_{i=1}^d\omega_t(x,x+e_i)[e^{-\tfrac{\alpha}{R^2}(1+2x_i)}+e^{-\tfrac{\alpha}{R^2}(1-2x_i)}-2]-\frac{\alpha}{\theta R^2}\right)\\
&\ge 
g(\hat x)\left(\sum_{i=1}^d\omega_t(x,x+e_i)[e^{-\alpha/R^2}(2+4\alpha^2x_i^2/R^4)-2]-\frac{\alpha}{\theta R^2}\right)\\
&\ge 
g(\hat x)(-C\frac{\alpha}{R^2}+c\frac{\alpha^2|x|^2}{R^4}-\frac{\alpha}{\theta R^2})
\\
&\ge 
\frac{\alpha}{R^2}g(\hat x)
(c\alpha\beta^2-C)>0
\end{align*}	
if  $\alpha$ is chosen to be large enough.
Hence $g(\hat X_t)$ is a submartingale for $t\le \tau_{\beta,1}$. 


Recall the definition of the stopping time $\Delta$ in \eqref{def:st-a-s}. Let
\[
v(\hat x):=\frac{g(\hat x)-e^{-\alpha}}{e^{-\alpha\beta^2}-e^{-\alpha}}
\quad \mbox{ and }
u(\hat x):=P_\omega^{\hat x}(X_{\Delta(B_R\setminus\bar B_{\beta R}, \theta R^2)}\in \partial B_{\beta R}).
\]
Set $\ms D=(B_R\setminus\bar B_{\beta R})\times[0,\theta R^2)$. 
Since $(u-v)|_{\partial^\p\ms D}\ge 0$ and $u(\hat X_t)$ is a martingale in $\ms D$, 
by the optional-stopping theorem we conclude that $u\ge v$ in $\ms D$.
In particular, 
$u(x,0)\ge v(x,0)\ge 
C(R^2-|x|^2)/R^2$ for $x\in B_R\setminus\bar B_{\beta R}$.
\end{proof}
%

\begin{lemma}\label{lem7}
Let $\beta\in (0,1)$,  and let $\tau_{\beta,1}$ be as in Lemma~\ref{lem6}. For $\theta>0$, there exists a constant $C=C(\beta,\kappa,d,\theta)$ such that, if $x\in B_R\setminus \bar B_{\beta R}\neq\emptyset$,
\[
P^{x,0}_\omega(X_{(\theta R^2)\wedge\tau_{\beta,1}}\notin\partial B_R)\le  C\dist(x,\partial B_R)/R.
\]
\end{lemma}
\begin{proof}
Set $\ms D:=(B_R\setminus \bar B_{\beta R})\times[0,\theta R^2)$. It suffices to consider the case $R>k^2$, where $k=k(\beta,\kappa,d,\theta)>\log 2/\log(2-\beta^2)$ is a large constant to be determined.
Let  $h(x,t)=2-|x|^2/(R+1)^2+t/(\theta R^2)$.

Recall the notation $\nabla^2$ in \eqref{eq:2nd-difference}. For $\hat x=(x,t)\in \ms D$, note that  $1\le h(\hat x)\le 3$ and $|\nabla_{e_i}^2(h^{-k})(\hat x)-\partial_{ii}(h^{-k})(\hat x)|\le Ck^3R^{-3}h^{-k}(x,t)$.
Hence for any $\hat x=(x,t)\in \ms D$, when $k$ is sufficiently large,
\begin{align*}
&\mc L_\omega (h^{-k})(\hat x)\\
&\ge 
\sum_{i=1}^d \omega_t(x,x+e_i)\partial_{ii}(h^{-k})-Ck^3R^{-3}h^{-k}+\partial_t (h^{-k})\\
&\ge
c\sum_{i=1}^d[k^2\tfrac{x_i^2}{(R+1)^4}h^{-k-2}+\tfrac{k}{(R+1)^2}h^{-k-1}]-Ck^3R^{-3}h^{-k}-\tfrac{k}{\theta R^2}h^{-k-1}\\
&\ge 
ckh^{-k}R^{-2}[k-C-Ck^2R^{-1}]>0,
\end{align*}
which implies that $h(\hat X_t)^{-k}$ is a submartingale inside the region $\ms D$.

Next, set (Recall the stopping time $\Delta$ in \eqref{def:st-a-s}.)
\[
u(\hat x)=P_\omega^{\hat x}(X_{\Delta(B_R\setminus\bar B_{\beta R}, \theta R^2)}\notin\partial B_R).
\]
Then $u(\hat X_t)+(2-\beta^2)^kh(\hat X_t)^{-k}$ is a submartingale in $\ms D$. Since
\[
\left\{
\begin{array}{rl}
&h^{-k}|_{x\in \partial B_R}\le (2-1+0)^{-k}=1\\
&h^{-k}|_{x\in\partial B_{\beta R}}\le (2-\beta^2)^{-k}\\
& h^{-k}|_{t=\theta R^2}\le (2-1+1)^{-k}\le (2-\beta^2)^{-k}
\end{array},
\right.
\]
by the optional stopping theorem, we have for $x\in B_R\setminus \bar B_{\beta R}$, 
\begin{align*}
u(x,0)+(2-\beta^2)^kh(x,0)^{-k}\le \sup_{\partial^\p\ms D} [u+(2-\beta^2)^kh^{-k}]\le (2-\beta^2)^k.
\end{align*}
Therefore, for any $x\in B_R\setminus \bar B_{\beta R}$,
\begin{align*}
u(x,0)&\le (2-\beta^2)^k(1-h(x,0)^{-k})\\
&\le C(h(x,0)-1)= C[1-|x|^2/(R+1)^2]\\
&\le C\dist(x,\partial B_R)/R.
\end{align*}
Our proof of Lemma~\ref{lem7} is complete.
\end{proof}

%% file: Apdx190814.tex
\newtheorem{atheorem}{Theorem}
\numberwithin{atheorem}{subsection}
\newtheorem{alemma}[atheorem]{Lemma}

\appendix
\section{Appendix}
\subsection{Properties (i)-(iii) in Remark~\ref{rm3}}
\begin{proof}
(i)Since $\mb Q$ is an invariant measure for $(\bar\omega_t)$, we have for any
bounded measurable function $f$ on $\Omega$, $y\in\Z^d$, $\hat x=(x,t)$, and $s<t$,
\begin{align}\label{inv-meas}
0&=E_\mb Q E_\omega^{0,0}[f(\theta_{-\hat x}(\bar\omega_{t-s}))-f(\theta_{-\hat x}\omega)]\nonumber\\
&=
E_\mb P\left[
\rho(\omega)\sum_{y\in\Z^d}p^\omega(0,0;x-y,t-s)[f(\theta_{-y,-s}\omega)-f(\theta_{-\hat x}\omega)]
\right]\nn\\
&=
E_\mb P\left[
f(\omega)[\sum_{y\in\Z^d}\rho(\theta_{\hat y}\omega)p^\omega(\hat y,\hat x)-\rho(\theta_{\hat x}\omega)]
\right],
\end{align}
where $\hat y=(y,s)$ and we used the translation-invariance of $\mb P$ in the last equality.  Moreover, by Fubini's theorem, for any bounded compactly-supported continuous function $\phi:\R\to\R$, 
\[
E_{\mb P}\left[f(\omega)
\int_{-\infty}^t\phi(s)[\sum_{y\in\Z^d}\rho(\theta_{\hat y}\omega)p^\omega(\hat y,\hat x)-\rho(\theta_{\hat x}\omega)]\dd s 
\right]
=0
\]
Thus we have that $\mb P$-almost surely, for any such test function $\phi$ on $\R$, 
\begin{equation*}
\int_{-\infty}^t\phi(s)[\sum_{y\in\Z^d}\rho_\omega(\hat y)p^\omega(\hat y,\hat x)-\rho_\omega(\hat x)]\dd s =0,
\end{equation*}
which (together with the translation-invariance of $\mb P$) implies that $\mb P$-almost surely, $\rho_\omega(x,t)\delta_x\dd t$ is an invariant measure for the process $(\hat X_t)_{t\ge 0}$.

(ii)
We have $\rho_\omega>0$ since the measures $\mb Q$ and $\mb P$ are equivalent. 
The uniqueness follows from the uniqueness of $\mb Q$ in \cite[Theorem 1.2]{DGR15}.

(iii) By \eqref{inv-meas} and Fubini's theorem, we also have that $\mb P$-almost surely, for any test function $\phi(t)$ as in (i) and any $h>0$, $x\in\Z^d$,
\[
\int_{-\infty}^\infty \phi(t)[\sum_{y\in\Z^d}\rho_\omega(y,t)(p^\omega(y,t;x,t+h)-\delta_x(y))-(\rho_\omega(x,t+h)-\rho_\omega(\hat x)]\dd t
=0.
\]
Dividing both sides by $h$ and letting $h\to 0$, we obtain \eqref{rho-invariance} with $\partial_t\rho_\omega$ replaced by the weak derivative. 
Note that the weak differentiability of $\rho_\omega$ in $t$ implies that it has an absolutely continuous (in $t$) version.
 Since $\rho_\omega$ is only used as a density, we may always assume that $\mb P$-almost surely, $\rho_\omega(x,\cdot)$ is continuous and almost-everywhere differentiable in $t$. 
\end{proof}

\subsection{Proof of Corollary~\ref{cor:hoelder}}
\begin{proof}
Assume $(x_0,t_0)=(0,0)$ and fix $R>0$. Let $R_k=2^{-k}R$ and $Q^k=B_{R_k}\times(-R_k^2,0]$. Note that $Q^{k+1}\subset Q^k$. For any bounded subset $E\subset\Z^d\times\R$, denote $\osc_E u:=\sup_E u-\inf_E u$. Set 
\[
v_k:=(u-\inf_{Q^k}u)/\osc_{Q^k}u.
\]
Notice that $\inf_{Q^k}v_k=0, \sup_{Q^k}v_k=1$ and 
\[
\osc_{Q^{k+1}}u=\osc_{Q^{k+1}}v_k \cdot\osc_{Q^k}u.
\]
We claim that $\osc_{Q^{k+1}}v_k\le 1-\delta$ for some $\delta=\delta(d,\kappa)\in(0,1)$. Indeed, replacing $v_k$ by $1-v_k$ if necessary, we can assume $\sup_{B_{R_{k+1}}\times(-\tfrac34 R_k^2,-\tfrac12 R_k^2)}v_k\ge 1/2$. By the PHI for $\mc L^*_\omega$ (Theorem~\ref{thm-ah}), 
\[
\inf_{Q^{k+1}}v_k\ge c\sup_{B_{R_{k+1}}\times(-\tfrac34 R_k^2,\tfrac12 R_k^2)}v_k\ge \tfrac c2:=\delta\in(0,1)
\]
and so $\osc_{Q^{k+1}}v_k\le \sup_{Q^k}v_k-\inf_{Q^{k+1}}v_k\le 1-\delta$. The claim is proved. So 
\[
\osc_{Q^{k+1}}u\le (1-\delta)\osc_{Q^k}u.
\]
If $r>R/2$,  the Corollary is trivial. If $r\le R/2$, we iterate the above inequality $k=\floor{\log_2(R/r)}$ times (so that $Q^{k+1}\subset B_r\times(-r^2,0]\subset Q^k$) to obtain 
\[
\osc_{B_r\times(-r^2,0]}u\le \osc_{Q^{k}}u\le (1-\delta)^k\osc_{Q^0}u\le (1-\delta)^{-1}(r/R)^\gamma \osc_{Q^0}u.
\]
where $\gamma=-\log_2(1-\delta)$. Our proof is complete.
\end{proof}

\subsection{Parabolic maximum principle}
In what follows we will prove a maximum principle for parabolic difference operators under the discrete space and continuous time setting. 
For any $\ms D\subset B_R\times(0,T)$, $\hat x:=(x,t)\in\ms D$ and  $u: \ms D\cup\partial^\p\ms D\to\R$, 
define
\[
I_u(\hat x)
:=\{
p\in\R^d: u(x,t)-u(y,s)\ge p\cdot(x- y)\mbox{ for all $(y,s)\in\ms D\cup\partial^\p\ms D$ with $s> t$}\},
\]
\[
\Gamma=\Gamma(u,\ms D):=\{(x,t)\in\ms D: I_u(x,t)\neq\emptyset\},
\]
\begin{align}\label{eq:181222}
\Gamma^+
=\Gamma^+(u,\ms D)=\{\hat x\in\Gamma:R|p|<u(\hat x)-p\cdot x \mbox{ for some }p\in I_u(\hat x)\}.
\end{align}
\begin{atheorem}[Maximum principle]\label{thm:max-p}
Let $\omega\in\Omega_\kappa$. Recall $\int_{\ms D}$ in \eqref{eq:def-integral}. Assume that $\ms D\subset B_R\times(0,T)$ is an open subset of $\Z^d\times\R$ for some $R,T>0$. Let $f$ be a measurable function on $\ms D$. For any function $u:\ms D\cup\partial^\p\ms D\to\R$  that solves $\mc L_a u\ge -f \quad{ in }\,{\ms D}$,
we have
\[
\sup_{\ms D} u\le \sup_{\partial^\p\ms D}u+ CR^{d/(d+1)}
(\int_{\Gamma^+}|f|^{d+1})^{1/(d+1)}
\]
\end{atheorem}
\begin{proof}
Without loss of generality, assume $f\ge 0$, $\sup_{\partial^\p\ms D}u=0$, and
\[\sup_{\ms D}u:=M>0.\]

Let 
\[
\Lambda=\{
(\xi,h)\in\R^{d+1}: R|\xi|<h<M/2
\}.
\]
For $(x,t)\in\ms D$, define a set
\[
\chi(x,t)=\{(p,u(x,t)-x\cdot p): p\in I_u(x,t)\}\subset \R^{d+1}.
\]

First, we claim that
\begin{equation}\label{e3}
\Lambda\subset \chi(\Gamma^+):=\bigcup_{(x,t)\in\Gamma^+}\chi(x,t).
\end{equation}
This will be proved by showing that for any $(\xi,h)\in\Lambda$, we have $(\xi,h)\in\chi(x_1,t_1)$ for some $(x_1,t_1)\in\Gamma^+$.
Indeed, fix $(\xi,h)\in\Lambda$ and define
\[
\phi(x,t):=u(x,t)-\xi\cdot x-h.
\]
Since $\sup_{\ms D}\phi\ge M-|\xi|R-h>0$, there exists $(x_0,t_0)\in\ms D$ with $\phi(x_0,t_0)>0$. Now for any $x\in\Z^d$, set (with the convention $\sup\emptyset=-\infty$)
\[
N_x=\sup\{t:(x,t)\in\ms D \text{ and } \phi(x,t)\ge 0\},
\]
and let $(x_1,t_1)$ be such that
\[
t_1=N_{x_1}=\max_{x\in B_R}N_x\ge N_{x_0}\ge t_0.
\]
By the continuity of $\phi$, we get $\phi(x_1,t_1)\ge 0$ and $(x_1,t_1)\in \ms D\cup\partial^\p\ms D$. Since  $\phi|_{\partial^\p\ms D}< 0$, we have $(x_1,t_1)\in\ms D$.
Moreover, since $\ms D$ is an open set, we can conclude that $\phi(x_1,t_1)=0$ and $\phi(x_1,s)<0$ for all $s>t_1$ with $(x_1,s)\in\ms D\cup\partial^\p\ms D$. Hence $\xi\in I_u(x_1,t_1)$ and $u(x_1,t_1)-\xi\cdot x_1=h>R|\xi|$, which implies that $(\xi, h)\in\chi(x_1,t_1)$ and $(x_1,t_1)\in\Gamma^+$. Display \eqref{e3} is proved.

Next, setting 
\[
\chi(\Gamma^+, x):=\bigcup_{s:(x,s)\in\Gamma^+}\chi(x,s),
\]
we will show that
\begin{equation}\label{e4}
\vol_{d+1}(\chi(\Gamma^+, x))
\le 
C\int_0^T (f(x,t)/\varepsilon)^{d+1}1_{(x,t)\in\Gamma^+}\dd t,
\end{equation}
where $\vol_{d+1}$ is the volume in $\R^{d+1}$.
To this end, let $
\tilde\chi(x,t)
=I_u(x,t)\times\{ u(x,t)\}\subset\R^{d+1}$.  Noting that, for fixed $x$, the map $(y,s)\mapsto(y,s+y\cdot x)$ is volume preserving, 
we then have
\begin{align}\label{e5}
\vol_{d+1}(\chi(\Gamma^+, x))
&=
\vol_{d+1}(\tilde\chi(\Gamma^+, x))\nonumber\\
&=
\int_0^T (-\partial_t u) \vol_d(I_u(x,t))1_{(x,t)\in\Gamma^+}\dd t.
\end{align}
For any  fixed $p\in I(x,t)$, $(x,t)\in\Gamma^+$, set
\[
w(y,s)=u(y,s)-p\cdot y.
\]
Then $I_w(x,t)=I_u(x,t)+p$. Since $w(x,t)-w(x\pm e_i,t)\ge \mp q_i$ for any $q\in I_w(x,t)$, $i=1,\ldots d$,
we have
\[
\vol_d(I_u(x,t))=\vol_d(I_w(x,t))
\le \prod_{i=1}^d[2u(x,t)-u(x+e_i,t)-u(x-e_i,t)].
\]
This inequality, together with \eqref{e5}, yields 
\begin{align*}
&\vol_{d+1}(\chi(\Gamma^+, x))\\
&\le -C\int_0^T \partial_t u\prod_{i=1}^d a_t(x,x+e_i)[2u(x,t)-u(x+e_i,t)-u(x-e_i,t)]1_{(x,t)\in\Gamma^+}\dd t\\
&\le 
C\int_0^T [-\mc L_a u(x,t)]^{d+1}1_{(x,t)\in\Gamma^+}\dd t.
\end{align*}
Display \eqref{e4} is proved. Finally, by \eqref{e3}, \eqref{e4} and $\vol_{d+1}(\Lambda)=CM^{d+1}/R^d$, 
we conclude that $
M^{d+1}/R^d\le C\int_{\Gamma^+}f^{d+1}$.
The theorem follows immediately. 
\end{proof}
\subsection{Mean value inequality}\label{subset:mvi}
\begin{atheorem}\label{thm:mvi}
Let $\theta>0$ and $a\in\Omega_\kappa$. Recall $\norm{\cdot}_{\ms D,p}$ in \eqref{eq:def-norm}. 
For any $\theta_1\in(0,\theta)$, $\rho\in (0,1)$ and $p>0$, there exists 
$C=C(\kappa,d,p,\theta,\theta_1,\rho)$ such that for any function $u$  that solves
$\mc L_a u\ge 0$ in $\ms D=B_R\times[0,\theta R^2)$,
we have
\[
\sup_{B_{\rho R}\times[0,\theta_1R^2)}u\le C\norm{u^+}_{\ms D,p}.
\]
\end{atheorem}
\begin{proof}
Since $\norm{u^+}_{\ms D,p}$ is increasing in $p>0$, it suffices to consider  $p\in(0,1)$.
Let $\beta\ge 2$ be a constant to be determined, and set
\[
\eta(x):=\big(1-\frac{|x|^2}{R^2}\big)^\beta 1_{x\in B_R},\quad
{
 \zeta(t):=\big(1-\dfrac{t}{\theta R^2}\big)^\beta 1_{0\le t<\theta R^2},}
\]
and set $v=\eta u^+$, $\bar v=v\zeta$. Define an elliptic operator $L_a^E$ to be
\[
L_a^E f(x,t)=\sum_{y:y\sim x}a_t(x,y)(f(y,t)-f(x,t)),
\]
so that $\mc L_a=L_a^E+\partial_t$. Note that $\bar v|_{\partial^\p \ms D}=0$ and $\bar v(\hat x)>0$ for $\hat x\in\Gamma^+(\bar v, \ms D)$. (Recall the definition of $\Gamma^+$ above Theorem~\ref{thm:max-p}.)
By the same argument as in \cite[displays (27),(28) and (29)]{GZ}, we have that on $\Gamma^+(\bar v, \ms D)$, $u^+=u$ and
\[
L_a^E v
\ge 
\eta L_a^E u-C_\beta\eta^{1-2/\beta}R^{-2}u^+.
\]
Hence, for $X=(x,t)\in \Gamma^+(\bar v, \ms D)$,
\begin{align*}
\mc L_a \bar v
&=\zeta L_a^E v+\partial_t(\zeta\eta u)\\
&\ge 
\zeta\eta L_a^E u-C_\beta\zeta\eta^{1-2/\beta}R^{-2}u^++\eta u\partial_t \zeta+\zeta\eta \partial_t u\\
&=
\zeta\eta\mc L_a u
-C_\beta\zeta\eta^{1-2/\beta}R^{-2}u^++\eta u\partial_t \zeta\\
&\ge 
-C_\beta\zeta\eta^{1-2/\beta}R^{-2}u^++\eta u\partial_t \zeta.
\end{align*}
Noting that in $\ms D$, $\partial_t\zeta\ge -C\beta R^{-2}\zeta^{1-1/\beta}/\theta$ and $\zeta,\eta\in[0,1]$, we have
\[
\mc L_a\bar v
\ge 
-C(\eta\zeta)^{1-2/\beta}R^{-2}u^+
\qquad\mbox{in }\Gamma^+(\bar v, \ms D).
\]
Applying Theorem~\ref{thm:max-p} to $\bar v$ and taking $\beta=2(d+1)/p$,
\begin{align*}
\sup_{\ms D}\bar v
&\le C\norm{(\eta\zeta)^{1-2/\beta}u^+/\varepsilon}_{\ms D,d+1}\\
&\le C(\sup_{\ms D}\bar v)^{1-p/(d+1)}\norm{(u^+)^{p/(d+1)}/\varepsilon}_{\ms D,d+1}.
\end{align*}
Since $\sup_{B_{\rho R}\times[0,\theta_1R^2)}u\le C\sup_{\ms D}\bar v$, the theorem follows.
\end{proof}

\subsection{Reverse H\"older implies $A_p$}
Recall $\abs{\ms D}, \int_{\ms D}, \norm{\cdot}_{\ms D,p}$, and the parabolic cubes in \eqref{eq:def-integral}, \eqref{eq:def-norm} and \eqref{eq:def-parakub}. 
\begin{alemma}
Let $K^0\subset\Z^d\times\R$ be a parabolic cube with side-length $r>0$. If a function $w>0$ on $K^0$ satisfies $RH_q(K^0)$, $q>1$, then 
\begin{enumerate}[(i)]
\item $w\in A_p(K^0)$ for some $1<p<\infty$;
\item $\frac{w(E)}{w(K)}\ge C(\frac{|E|}{|K|})^c$ for all $E\subset K$ where $K\neq\emptyset$ is a subcube of $K^0$.
\end{enumerate}
\end{alemma}
\begin{proof}
First, we claim that there exist constants $\gamma,\delta\in (0,1)$ such that $w(E)>\gamma w(K)$ implies $|E|>\delta|K|$ for all $E\subset K$ where $K\neq\emptyset$ is a subcube of $K^0$. Indeed, this is a simple consequence of H\"older's inequality:
\begin{align*}
\frac{1}{|K|}w(E)=\frac{1}{|K|}\int_{K} w\mathbbm{1}_{E}
\le (\frac{|E|}{|K|})^{1/q'}\norm{w}_{K,q}\stackrel{(RH_q)}{\le} C\left(\frac{|E|}{|K|}\right)^{1/q'}\frac{w(K)}{|K|},
\end{align*}	
where $q'=q/(q-1)$ denotes the conjugate of $q$. 

Assume $K^0=K_r$. Let $\mc M_k(K_r), k>1$ be the family of nonempty subcubes of $K_r$ of the form 
\[
(\prod_{i=1}^d[\frac{m_i}{2^k}r,\frac{1+m_i}{2^k}r)\cap\Z^d)\times[\frac{n_i}{4^k}r^2,\frac{1+n_i}{4^k}r^2)
\]
where $m_i,n_i$'s are integers. Elements in  $M_k(K_r)$ are called  $k$-level dyadic subcubes of $K_r$. Note that every $k$-level cube $K$ is contained in a unique $(k-1)$-level ``parent" denoted by $K^{-1}$.
 Since the class $A_p$ is invariant under constant multiplication, we may assume that $w(K^0)/|K^0|=1$.

Let $f:=w^{-1}\mathbbm{1}_{K^0}$ and define a maximal function
\[
M_f(x):=\sup_{K\ni x}\frac{1}{w(K)}\sum_K |f|w,
\] where the supremum is taken over all dyadic subcubes $K$ of $K_0$.
 Consider the level sets
 \[
 E_k=\{x\in K^0: M_f(x)>2^{Nk}\}, \quad k=0,1,2,\ldots
 \]
 where $N$ is a big constant to be determined. Notice that by assumption, $E_0$ is comprised of dyadic subcubes strictly smaller than $K_0$. Since $w$ is volume-doubling, there exists a constant $c_0>0$ such that for any maximal dyadic subcube $K$ of $E_{k-1}$,
 \[
 \int_K fw\le \int_{K^{-1}}fw\le 2^{N(k-1)}w(K^{-1})\le 2^{N(k-1)+c_0}w(K).
 \]
Moreover, for the same $K$, we have $2^{Nk}w(E_k\cap K)\le \int_{K}fw$ and so, by the inequality above, $w(E_k\cap K)\le 2^{c_0-N}w(K)$.
We now take $N$ to be large enough that $w(E_k\cap K)\le (1-\gamma) w(K)$ which implies $|E_k\cap K|\le (1-\delta)|K|$.   Summing over all such $K$'s, we have  $|E_k|\le (1-\delta) |E_{k-1}|$, $k\ge 1$.
 Thus
\[
|E_k|\le\delta^k|E_0|\le \delta^k|K^0|, \quad k=0,\ldots
\]
and so, for $p>1$ chosen so that $p'=p/(p-1)$ is sufficiently close to $1$,
\begin{align*}
\int_{K^0} f^{p'-1}
&=\int_{K^0\cap\{x:M_f\le 1\}}f^{p'-1}+\sum_{k=0}^\infty \int_{E_k\setminus E_{k+1}}f^{p'-1}\\
&\le |K^0|+\sum_{k=0}^\infty 2^{(p'-1)N(k+1)}\delta^k|K^0|\\
&\le C|K^0|.
\end{align*}
(i) is proved. (ii) then follows from  H\"older's inequality
\begin{align*}
\frac{1}{w(K)}\int_E w^{-1}w\le \left(\frac{1}{w(K)}\int_K w^{-p'}w\right)^{1/p'}\left(\frac{1}{w(K)}\int_K\mathbbm{1}_E w\right)^{1/p}
\end{align*}	
and the $A_p$ inequality.
\end{proof}

\subsection{Proof of Corollary~\ref{cor:q-estimates}\eqref{cor:green1}\eqref{cor:green2}}
\begin{proof}

\eqref{cor:green1} For any $\hat x=(x,t)\in\R^d\times[0,\infty)$ and $\omega\in\Omega_\kappa$,
set
\[
v(\hat x)=q^\omega(\hat 0; \floor{x},t) \quad\mbox{ and }\quad
a^\omega(x):=\int_0^\infty(v(0,t)-v(x,t))\dd t.
\] 
When $d=2$, it suffices to consider $x\in\mb B_1\setminus\{0\}$.
We fix a small number $\epsilon\in(0,1)$ and split the integral $a^\omega(nx)$ into four parts:
\begin{align*}
a^\omega(nx)=\int_0^{n^\epsilon}+\int_{n^\epsilon}^{n^2}
+\int_{n^2}^\infty=:\rom{1}+\rom{2}+\rom{3},
\end{align*}
where it is understood that the integrand is $(v(0,t)-v(x,t))\dd t$. 

First, we will show that $\mb P$-almost surely,
\begin{equation}\label{eq:green-1}
\varlimsup_{n\to\infty}|\rom{1}|/\log n\le \epsilon.
\end{equation}
By Theorem~\ref{thm:hke},   for any $t\in(0,n^\epsilon)$, $x\in\Z^2\setminus\left\{0\right\}$ and all $n$ large enough,
$v(nx,t)
\le Ce^{-cn|x|}/\rho_\omega(B_{\sqrt t},0).
$
Thus
\[
\int_0^{n^\epsilon}v(nx,t)\dd t
\le 
\frac{n^\epsilon}{\rho_\omega(\hat 0)} e^{-cn|x|}.
\]
By \eqref{cor:q-hke}, there exists   $t_0(\omega)>0$ such that for $n$ big enough with $n^\epsilon>t_0$, 
\begin{align*}
\int_0^{n^\epsilon}v(0,t)\dd t
\le 
\frac{Ct_0}{\rho_\omega(\hat 0)}+\int_{t_0}^{n^\epsilon}\frac{C}{t}\dd t
\le \frac{Ct_0}{\rho_\omega(\hat 0)}+C\epsilon\log n,
\end{align*}
Display \eqref{eq:green-1} follows immediately.

In the second step, we will show that (note that $2p^\Sigma_1(0,0)=1/\pi\sqrt{\det\Sigma}$)
\begin{equation}\label{eq:green-2}
\limsup_{n\to\infty}|\rom{2}-2p_1^\Sigma(0,0)\log n|/\log n\le C\epsilon,
\quad \mb P\mbox{-a.s.}
\end{equation}
Indeed, by Theorem~\ref{thm:llt}, there exists $C(\omega,\epsilon)>0$ such that
$|tv(0,t)-p_1^\Sigma(0,0)|\le\epsilon$ whenever $t\ge C(\omega,\epsilon)$.
Now, taking $n$ large enough such that $n^{\epsilon}>C(\omega,\epsilon)$,
\begin{align}\label{eq:green-21}
&\Abs{\int_{n^\epsilon}^{n^2}v(0,t)\dd t-(2-\epsilon)p_1^\Sigma(0,0)\log n}\nn\\
&\le 
\int_{n^\epsilon}^{n^2}\Abs{\frac{tv(0,t)-p^\Sigma_1(0,0)}{t}}\dd t\nn\\
&\le 
\epsilon\int_{n^\epsilon}^{n^2}\frac{\dd t}{t}<2\epsilon\log n.
\end{align}
On the other hand, for $t\ge n^\epsilon>t_0(\omega)$, by \eqref{cor:q-hke},
$v(nx,t)\le \frac{C}{t}(e^{-cn|x|}+e^{-cn^2|x|^2/t})$. Thus
\begin{equation}\label{eq:green-22}
\int_{n^\epsilon}^{n^2}v(nx,t)\dd t\le 
\int_{n^\epsilon}^{n^{2-\epsilon}}\frac{C}{t}e^{-cn^\epsilon|x|^2}\dd t
+\int_{n^{2-\epsilon}}^{n^2}\frac{C}{t}\dd t
\le C\epsilon\log n.
\end{equation}
Displays \eqref{eq:green-21} and \eqref{eq:green-22} imply \eqref{eq:green-2}.

Finally, we will prove that for $\mb P$-almost every $\omega$,
\begin{equation}\label{eq:green-3}
\limsup_{n\to\infty}|\rom{3}|/\log n=0.
\end{equation}
Since $|x|<1$, by \eqref{eq:holder-v}, for any $t\ge n^2\ge t_0(\omega)$,
\[
\Abs{v(0,t)-v(nx,t)}\le 
C\left(
\frac{n}{\sqrt t}
\right)^\gamma t^{-1}.
\]
Therefore, $\mb P$-almost surely, when $n^2>t_0(\omega)$,
\begin{align*}
\Abs{\int_{n^2}^\infty v(0,t)-v(nx,t)\dd t}
&\le 
Cn^\gamma\int_{n^2}^\infty \frac{1}{t^{\gamma/2+1}}\dd t\le C.
\end{align*}
Display \eqref{eq:green-3} follows. 
Combining \eqref{eq:green-1}, \eqref{eq:green-2} and \eqref{eq:green-3}, we have for $d=2$, 
\[
\varlimsup_{n\to\infty}\Abs{\frac{a^\omega(nx)}{\log n}-2p_1^\Sigma(0,0)}\le C\epsilon,
\]
Noting that $\epsilon>0$ is arbitrary, we obtain Corollary~\ref{cor:q-estimates}\eqref{cor:green1}.

\eqref{cor:green2} 
We fix a small constant $\epsilon\in(0,1)$.  Note that 
\[
n^{d-2}\int_0^\infty q^\omega(\hat 0;\floor{nx},t)\dd t
=\int_0^\infty n^d v(nx,n^2s)\dd s.
\]
For any fixed $x\in\R^d$, write
\begin{align*}
\int_0^\infty n^d v(nx,n^2s)\dd s
=\int_0^{n^{-\epsilon}}+\int_{n^{-\epsilon}}^\epsilon+\int_\epsilon^{1/\sqrt \epsilon}+\int_{1/\sqrt\epsilon}^\infty
=:\rom{1}+\rom{2}+\rom{3}+\rom{4}.
\end{align*}
First, by Theorem~\ref{thm:hke}, for $s\in(0,n^{-\epsilon})$, 
we have $v(nx,n^2s)\le Ce^{-cn^\epsilon|x|^2}/\rho_\omega(\hat 0)$, hence
\begin{equation}\label{eq:green2-1}
\varlimsup_{n\to\infty}\rom{1}\le C\lim_{n\to\infty}n^{d-\epsilon}e^{-cn^\epsilon|x|^2}/\rho_\omega(\hat 0)=0.
\end{equation}
Second, by \eqref{cor:q-hke}, when $n$ is large enough, then for all $t\ge n^{2-\epsilon}$, we have 
$v(nx,t)\le Ct^{-d/2}e^{-cn^2|x|^2/t}$ . Hence
\begin{equation}\label{eq:green2-2}
\varlimsup_{n\to\infty}\rom{2}\le 
\varlimsup_{n\to\infty} Cn^d\int_{n^{-\epsilon}}^\epsilon (n^2s)^{-d/2}e^{-c|x|^2/s}\dd s\le C\epsilon.
\end{equation}
Moreover, by Theorem~\ref{thm:llt}, there exists $N(\omega,\epsilon)$ such that for $n\ge N(\omega,\epsilon)$, we have $\sup_{|s|\ge \epsilon}|v(nx,n^2s)-p_s^\Sigma(0,x)|\le\epsilon$. Hence
\begin{equation}\label{eq:green2-3}
\varlimsup_{n\to\infty}\Abs{\rom{3}-\int_{\epsilon}^{1/\sqrt \epsilon}p_s^\Sigma(0,x)\dd s}\le \sqrt\epsilon.
\end{equation}
Further, by \eqref{cor:q-hke},  
for $d\ge 3$,
\begin{equation}\label{eq:green2-4}
\varlimsup_{n\to\infty}\rom{4}
\le 
C\int_{1/\sqrt\epsilon}^\infty \frac{n^d}{(n^2s)^{d/2}}\dd s=C\epsilon^{(d-2)/4}.
\end{equation}
Finally,  combining \eqref{eq:green2-1},\eqref{eq:green2-2}, \eqref{eq:green2-3} and \eqref{eq:green2-4}, we get
\[
\varlimsup_{n\to\infty}
\Abs{\int_0^\infty n^d v(nx,n^2s)\dd s-\int_{\epsilon}^{1/\sqrt \epsilon}p_s^\Sigma(0,x)\dd s}
\le C\epsilon^{1/4}.
\]
Letting $\epsilon\to 0$,  \eqref{cor:green2} is proved.
\end{proof}